\documentclass[11pt]{article}
\usepackage{amsmath,amsfonts,amsthm,amssymb}
\usepackage[margin=1in,letterpaper]{geometry}
\usepackage{physics}
\usepackage{graphicx}
\usepackage{blkarray}
\usepackage{kbordermatrix}
\usepackage{amsthm}
\theoremstyle{definition}
\newtheorem{defi}{Definition}[section]
\newtheorem{theorem}[defi]{Theorem}
\newtheorem{corollary}[defi]{Corollary}
\newtheorem{lemma}[defi]{Lemma}
\newtheorem{prop}[defi]{Proposition}
\newtheorem{conj}[defi]{Conjecture}
\newtheorem{claim}[defi]{Claim}
\newtheorem{eg}[defi]{Example}
\usepackage{bm}

\newcommand{\bx}{\bm x}
\newcommand{\by}{\bm y}
\newcommand{\bp}{\bm p}
\newcommand{\bq}{\bm q}
\newcommand{\bc}{\bm c}
\newcommand{\bd}{\bm d}

\title{Generic Global Rigidity in $\ell_p$-Space and the Identifiability of the $p$-Cayley-Menger Varieties}
\author{Tomohiro Sugiyama and Shin-ichi Tanigawa\thanks{Department of Mathematical Informatics, Graduate School of Information Science and Technology, University of Tokyo, 7-3-1 Hongo, Bunkyo-ku, 113-8656,  Tokyo Japan. email: {\tt working-sugiyama@mist.i.u-tokyo.ac.jp, tanigawa@mist.i.u-tokyo.ac.jp}}}

\begin{document}

\maketitle

\begin{abstract}
The celebrated result of Gortler-Healy-Thurston (independently, Jackson-Jord{\'a}n for $d=2$) shows that the global rigidity of graphs realised in the $d$-dimensional Euclidean space is a generic property. 
Extending this result to the global rigidity problem in $\ell_p$-spaces remains an open problem.
In this paper we affirmatively solve this problem  when $d=2$ and $p$ is an even positive integer.

A key tool in our proof is a sufficient condition for the $d$-tangential weak non-defectivity of projective varieties due to Bocci, Chiantini, Ottaviani, and Vannieuwenhoven.  
By specialising the condition to the $p$-Cayley-Menger variety, which is the $\ell_p$-analogue of the Cayley-Menger variety for Euclidean distance, we provide an $\ell_p$-extension of the generic global rigidity theory of Connelly.
As a  by-product of our proof, we also offer a purely graph-theoretical characterisation of the $2$-identifiability of an orthogonal projection of the $p$-Cayley-Menger variety along a coordinate axis of the ambient affine space. 
\end{abstract}
\medskip \noindent {\bf Keywords:} graph rigidity, generic global rigidity, $\ell_p$-spaces, Cayley-Menger varieties, identifiability, tangential weak defectivity

\section{Introduction}
Generic rigidity is the central idea in graph rigidity theory. A classical result due to Gluck~\cite{gluck} and Asimov-Roth~\cite{asimow1978rigidity} states that local rigidity is a generic property, meaning that a graph is 
either locally rigid for almost all locations of vertices in the $d$-dimensional Euclidean space or 
flexible  for almost all locations of vertices in the same space.  
A deeper result due to Gortler-Healy-Thurston~\cite{gortler2010characterizing} (independently, Jackson-Jord{\'a}n~\cite{jackson2005connected} for $d=2$)
shows that global rigidity is also a generic property.
This fact is a cornerstone in recent developments of graph rigidity theory (see, e.g.,~\cite{JW}), and since then  there have been many efforts to extend it  to a wide range of global rigidity problems. However the current achievements in this research direction are not so extensive, and a generic property is known only for variants of Euclidean rigidity such as global rigidity on a cylinder~\cite{jackson2019}, on a flat torus~\cite{KST}, on linear subspaces~\cite{guler2021global,cruickshank2024global}, in pseudo-Euclidean spaces~\cite{GT}, or global rigidity problems under linear measurements such as direction rigidity. 
There are also negative results for other measurement functions such as Euclidean inner product~\cite{JJT16} or Euclidean volume~\cite{Southgate}, where the corresponding global rigidity problem may not necessarily be a generic property. 
Currently, there is no theory that characterizes which measurement functions give a global rigidity problem with the generic property.

In this paper, we consider the global rigidity problem in $\ell_p$-spaces, and show that global rigidity is a generic property in the $\ell_p$-plane if $p$ is an even positive integer.
To the best of our knowledge, this is the first generic global rigidity result for a high-degree polynomial measurement. 
The major difficulty in extending the result of Euclidean rigidity to a general $\ell_p$-space is the lack of the theory of stress matrices.
This theory, originated by Connelly~\cite{connelly2005generic}, extensively uses the fact that the squared Euclidean distance is a quadratic polynomial, and there is no straightforward counterpart argument for non-quadratic distance functions.
We will overcome this difficulty using a tool in \cite{cruickshank2023identifiability}.
Cruickshank, Mohammadi,  Nixon, and the second author~\cite{cruickshank2023identifiability} have recently established a close connection between the global rigidity problem under a general measurement and the identifiability problem of secant varieties. 
In this paper, we provide a more detailed exposition of this connection by specialising our focus to the $p$-Cayley-Menger variety, which is the $\ell_p$-analogue of the well-known Cayley-Menger variety for Euclidean distance.
It turns out that, when we apply a sufficient condition for the $d$-tangential weak non-defectivity  (and the $d$-identifiability) by Borcci, Chiantini, Ottaviani, and Vannieuwenhoven~\cite{BCC,chiantini2014algorithm} to the $p$-Cayley-Menger variety, we obtain an $\ell_p$-analogue of Connelly's sufficient condition.
We show that this condition is also necessary in the $\ell_p$-plane by using the inductive construction of 2-connected redundantly locally rigid graphs by Dewar, Hewetson, and Nixon~\cite{dewar2024uniquely}.

A by-product of our proof is a purely graph-theoretical characterisation of the $2$-identifiability of an axis-parallel projection of the $p$-Cayley-Menger variety. 
While the identifiability of projective varieties has been an active topic in algebraic geometry,  
there seems to be little knowledge on how the identifiability changes when a non-generic projection is applied.
We give a combinatorial and efficiently testable characterisation for when  the $2$-identifiability of  the $p$-Cayley-Menger variety is preserved by a projection along a coordinate axis of the ambient affine space.

Let us also review previous results on graph rigidity in $\ell_p$-space.
The theory of local rigidity in $\ell_p$-space was initiated by Kitson and Power~\cite{kitson2014infinitesimal}, who provided a neat characterisation of local rigidity in the $\ell_p$-plane. 
The result of Kitson and Power has been extended to more general normed spaces~\cite{D21,K15}, 
and the higher dimensional case has been also studied~\cite{dewarkitson2022}.
Global rigidity has not been well understood,
but recently Dewar, Hewetson, and Nixon~\cite{dewar2024uniquely,dewarnixon2022} gave a combinatorial characterisation of graphs that can be globally rigid in a generic realisation in an analytic normed space.
However, it is still not known whether the global rigidity is a generic property beyond Euclidean space. 
We also note that having the generic property provides a significant advantage in the detailed analysis of global rigidity properties, such as global linkedness, see, e.g.,~\cite{jordan2024globally} for results in the Euclidean case.

The paper is organised as follows.
In Section~\ref{sec:rigidity},
we introduce the basic notation and terminology on graph rigidity and then formally state our results.
In Section~\ref{sec:iden}, we explain  the identifiability and tangential weak defectivity, key new tools in our global rigidity analysis,
and prove a Connelly-type sufficient condition for generic global rigidity in general dimensions.
As an application of the tools introduced in Section~\ref{sec:iden}, in Section~\ref{sec:suff} we give a sufficient condition for generic global rigidity in general dimensions in terms of generic local rigidity.
In Section~\ref{sec:plane}, we prove our main result, a characterisation of generic global rigidity in the $\ell_p$-plane.

\medskip
Throughout the paper we use the following basic notations.
We consider a $d$-dimensional real vector space equipped with $\ell_p$-norm,
where the distance between two points 
$\bx=(x_1,\dots, x_d), \by=(y_1,\dots, y_d)\in \mathbb{R}^d$ is given by $\|\bx-\by\|_p:=\sqrt[p]{\sum_{i=1}^d (x_i-y_i)^p}$.
This metric space is called the $d$-dimensional $\ell_p$-space and is denoted by $\ell_p^d$.
When $p=2$, it is a Euclidean space.

Let $\mathbb{F}\in \{\mathbb{R},\mathbb{C}\}$.
For a point or a vector $\bx\in \mathbb{F}^d$,
$(\bx)_k$ denotes the $k$-th coordinate of $x$.
For a point $\bx = (x_1, \ldots, x_d) \in \mathbb{F}^d$ and $k \in (0, \infty)$, we write
$\bx^{k} = (x_1^k, \ldots, x_1^k)$.
A point $\bx$ in $\mathbb{F}^d$ is sometimes regarded as a matrix of size $1\times d$.

For a vector space $W$ and $X\subseteq W$, let $\langle X\rangle$ be the linear span of $X$ in $W$.
For a matrix $A$ and positive integers $i,j$,
the $(i,j)$-th entry of $A$ is denoted by $A[i,j]$.

\section{Rigidity in $\ell_p$-space and main theorems}\label{sec:rigidity}
In this section, we shall first review preliminary results on graph rigidity and then state our main results.
\subsection{Local and global rigidity}\label{subsec:local and global}

A $d$-dimensional {\em framework} is a pair $(G,\bp)$ of a graph $G=(V,E)$ and a map $\bp:V\rightarrow \mathbb{R}^d$.
The map $\bp$ is often referred to as a {\em point configuration}, and it is sometimes regarded as a point in $(\mathbb{R}^d)^n$, where $n=|V|$.
Two $d$-dimensional frameworks $(G,\bp)$ and  $(G,\bq)$ with the same underlying graphs  are said to be {\em equivalent} if 
\[
\|\bp(i)-\bp(j)\|_p=\|\bq(i)-\bq(j)\|_p \qquad (ij\in E(G)), 
\]
and they are {\em congruent} if 
there is an isometry $\iota$ of $\ell_p^d$ such that 
$\bq(i)=\iota (\bp(i))$ for all $i\in V(G)$.
When $p\neq 2$, $\iota$ is 
a translation, a permutation of coordinates,
or a reflection along a coordinate axis.

A framework $(G,\bp)$ is {\em globally rigid} in $\ell_p^d$
if every  framework equivalent to $(G,\bp)$ is congruent,
and it is called {\em locally rigid} if this happens in a small neighbour of $\bp$ in $(\mathbb{R}^d)^{n}$.
An equivalent way to define local rigidity is in terms of continuous motions: a framework $(G,\bp)$ is not locally rigid if and only if 
there is a continuous map $t \mapsto \bp^t\in (\mathbb{R}^d)^{n}$ for $t\in [0,1]$
such that $\bp^0=\bp$ and $(G,\bp^t)$ is equivalent but not congruent to $(G,\bp)$ for $t>0$. See, e.g.,~\cite{kitson2014infinitesimal} for more details.

A key tool to analyse rigidity is the following $\ell_p$-analogue of the rigidity map (or the measurement map).
For a positive integer $d$ and a real number $p$ with $1<p<\infty$, define $f_p^{\times d}: (\mathbb{R}^d)^{n}\rightarrow \mathbb{R}^{n\choose 2}$ by 
\begin{equation}\label{eq:f_d}
(f_p^{\times d}(\bp))_{ij}=\sum_{k=1}^d ( (\bp(i))_k-(\bp(j))_k)^p \qquad (\bp\in (\mathbb{R}^d)^{n}).
\end{equation}
For a graph $G=(V,E)$, define the projection map $\pi_G:\mathbb{R}^{n\choose 2}\rightarrow \mathbb{R}^{E}$ along coordinate axes onto $\mathbb{R}^E$, which is the subspace of $\mathbb{R}^{n\choose 2}$ indexed by the elements in $E$.
Then $\pi_G\circ f_p^{\times d}(\bp)$ is the list of the $p$-th powered edge lengths of $(G,\bp)$ when $p$ is an even integer.
Hence we have the following.
\begin{prop}\label{prop:reformulation}
Let $p$ is an even positive integer with $p \neq 2$ and $(G,\bp)$ be a framework in $\ell_p^d$.
Then $(G,\bp)$ is globally rigid if and only if 
any $\bq$ in the fibre of  $\pi_G\circ f_p^{\times d}(\bp)$ under $\pi_G\circ f_p^{\times d}$  can be obtained from $\bp$ by a translation, a permutation of coordinates, and the reflections along coordinate axes.
\end{prop}

\subsection{Generic local rigidity}\label{subsec:local}
It is an NP-hard problem to decide the local/global rigidity of a given framework in a Euclidean space, 
and a common tractable approach is to restrict our attention to generic frameworks. We say that a framework $(G,\bp)$ or a point-configuration  $\bp:V(G)\rightarrow \mathbb{R}^d$ is {\em generic} if the set of coordinates of $\bp$ is algebraically independent over $\mathbb{Q}$. 

For generic frameworks, local rigidity admits the following simple characterisation, which is the adaptation of a result of Gluck~\cite{gluck} or Asimov-Roth~\cite{asimow1978rigidity} to the $\ell_p$-setting. 

\begin{prop}(Kitson and Power \cite{kitson2014infinitesimal})\label{prop:inf}
Suppose $(G,\bp)$ is a generic framework and let $p$ be a real number with $1< p< \infty$ and  $p \neq 2$.
Then $(G,\bp)$ is locally rigid in $\ell_p^d$ if and only if 
$\rank J(\pi_G\circ f_p^{\times d})(\bp)=dn-d$.
\end{prop}

The equivalence between two properties in Proposition~\ref{prop:inf}
may not be valid if $(G,\bp)$ is not generic.
A framework $(G,\bp)$ is called {\em infinitesimally rigid} in $\ell_p^d$
if $\rank J(\pi_G\circ f_p^{\times d})(\bp)=dn-d$.

Proposition~\ref{prop:inf} implies that, for a  graph $G$ and a positive integer $d$,
either every generic framework $(G,\bp)$ is locally rigid in $\ell_p^d$ or 
no generic framework $(G,\bp)$ is locally rigid in $\ell_p^d$.
Hence, we may say that a graph $G$ is {\em locally rigid in $\ell_p^d$} if  a/every generic framework $(G,\bp)$ is locally rigid. Since this is a property of graphs, it is natural to ask a combinatorial characterization. For $d = 2$,  Kitson and Power gave a concrete answer. 
\begin{theorem}[Kitson and Power~\cite{kitson2014infinitesimal}]
Suppose $p$ is a real number with $1\leq  p\leq  \infty$ and  $p \neq 2$.
A graph $G$ is locally rigid in the $\ell_p$-plane if and only if 
$G$ is 2-tree-connected, i.e., $G$ contains two edge-disjoint spanning trees.
\end{theorem}
In general, a graph is said to be {\em $k$-tree-connected} if it contains $k$ edge-disjoint spanning trees.
Kitson and Power also conjectured that $d$-tree-connectivity characterises the $d$-dimensional local rigidity.
This problem is still open.

\subsection{Generic global rigidity in Euclidean space}\label{subsec:global in l2}
Before moving to global rigidity in $\ell_p$-space, we shall quickly review results in Euclidean space.

The central approach to the generic global rigidity problem is the theory of self-stresses due to Connelly~\cite{connelly2005generic}.
Given a $d$-dimensional framework $(G,\bp)$, a map $\omega:E(G)\rightarrow \mathbb{R}$ is called a {\em self-stress} (in Euclidean space) if 
it satisfies the following equilibrium condition:
\begin{equation}\label{eq:equilibrium}
\sum_{v:uv\in E(G)} \omega(ij)(\bp(j)-\bp(i))=0 \qquad (u\in V(G)). 
\end{equation}
A {\em stress matrix} of a framework $(G,\bp)$ is the Laplacian matrix of $G$ weighted by a self-stress. In general, for a graph $G$ with the vertex set $V=\{1,2,\dots, n\}$ and an edge weight $\omega:E(G)\rightarrow \mathbb{R}$, the weighted Laplacian $L_{G,\omega}$ is an $n\times n$ symmetric matrix with the entries defined by 
\[
L_{G,\omega}[i,j]=\begin{cases}
\sum_{k:ik\in E(G)} \omega(ik) & (i=j) \\
-\omega(ij) & (i\neq j, ij\in E(G)) \\
0 & (i\neq j, ij\notin E(G)).
\end{cases}
\]

Weighted Laplacians and the equilibrium condition (\ref{eq:equilibrium}) are closely related. Indeed, given $(G,\bp)$ and $\omega:E(G)\rightarrow \mathbb{R}$, $\omega$ is a self-stress of $(G,\bp)$ if and only if 
each coordinate vector of $(G,\bp)$ (that is, the $n$-dimensional vector obtained by aligning $(\bp(1))_k, (\bp(2))_k,\dots, (\bp(n))_k$) is in the kernel of $L_{G,\omega}$.
Moreover, any Laplacian has the all-one vector in its kernel.
Therefore, for any self-stress $\omega$ of $(G,\bp)$, 
the dimension of the kernel of $L_{G,\omega}$ is at least one plus the dimension of the affine span of the points of $\bp$.
In particular, if $(G,\bp)$ is generic and $n\geq d+1$, then 
$\rank L_{G,\omega}\leq n-(d+1)$.

\begin{theorem}[Gortler-Healy-Thurston~\cite{gortler2010characterizing}]\label{thm:GHT}
A generic framework $(G,\bp)$ in the $d$-dimensional Euclidean space with $n \geq d+2$ vertices
is globally rigid if and only if 
$(G,\bp)$ has a self-stress $\omega$ such that 
$\rank L_{G,\omega}=n-(d+1)$.
\end{theorem}

Since the condition of Theorem~\ref{thm:GHT} is algebraic in the coordinates of $\bp$, Theorem~\ref{thm:GHT} implies that Euclidean global rigidity is a generic property of graphs. 
Thus it enables us to extend the idea of graph local rigidity by Gluck, Asimov, and Roth to global rigidity, and we say that a graph $G$ is {\em globally rigid} in the $d$-dimensional Euclidean space if a/every generic framework of $G$ in $\ell_2^d$ is globally rigid. 
As with local rigidity, we can further ask a combinatorial characterization of  global rigidity. For $d=2$, this question has been answered by Jackson and Jord{\'a}n~\cite{jackson2005connected},
and the higher dimensional characterisation problem is still open.
(In fact, the proof of Jackson and Jord{\'a}n is independent from that of Gortler-Healy-Thurston in the sense that it also gives Theorem~\ref{thm:GHT} for two-dimensional frameworks.)
A further refined analysis of global rigidity in terms of self-stresses can be found in \cite{G23}.

\subsection{Global rigidity problem in $\ell_p$-space}\label{global in lp}
Characterising global rigidity is a challenging problem in general, but the following folklore fact states that the problem is easy in the 1-dimensional line.
\begin{prop}\label{prop:1d}
Let $(G,\bp)$ be a framework in the $\ell_p$-line with $n \geq 2$ vertices such that $\bp$ is injective.
\begin{itemize}
\item If $(G,\bp)$ is globally rigid, then $G$ is 2-connected.
\item If $(G,\bp)$ is generic and $G$ is 2-connected, then $(G,\bp)$ is globally rigid.
\end{itemize}
\end{prop}

The necessary condition for generic global rigidity in general dimensions has been shown by Dewar, Hewetson, and Nixon~\cite{dewar2024uniquely} by adapting Hendrickson's argument.
Although they have established the result in a more general non-Euclidean normed space, we only mention the result in the $\ell_p$-plane.
\begin{theorem}[Dewar, Hewetson, and Nixon~\cite{dewar2024uniquely}]\label{thm:necessity}
Let $p$ be an even positive integer with $p\neq 2$ 
and let $(G,\bp)$ be a generic framework in $\ell_p^2$ with $n \geq 2$ vertices.
If $(G,\bp)$ is globally rigid, then $G$ is 2-connected and 
$G$ is redundantly 2-tree-connected, i.e.,
$G-e$ is 2-tree-connected for every $e\in E(G)$.
\end{theorem}

The main theorem of the paper by Dewar, Hewetson, and Nixon~\cite{dewar2024uniquely} (which is also given in a general non-Euclidean normed plane) is that a graph satisfying the combinatorial necessary condition in Theorem~\ref{thm:necessity} can be realised as a globally rigid generic framework.
\begin{theorem}[Dewar, Hewetson, and Nixon~\cite{dewar2024uniquely}]\label{thm:DHN}
Let $p$ be an even positive integer with $p\neq 2$ and $G=(V,E)$ be a graph with $n \geq 2$ vertices.
Then there exists a globally rigid generic framework $(G,\bp)$ in $\ell_p^2$
if and only if $G$ is 2-connected and $G$ is redundantly 2-tree-connected.
\end{theorem}

\subsection{Main rigidity theorems}\label{subsec:main}
The main result of this paper is the $\ell_p$-counterpart of Theorem~\ref{thm:GHT} when $d=2$ and $p$ is even. 
The result in particular implies that 2-dimensional global $\ell_p$-rigidity is a generic property of graphs.

To state our theorems, 
we need to introduce a new notion about self-stresses.
Let $p$ be a positive integer with $p\geq 2$.
For a $d$-dimensional framework $(G,\bp)$, 
we define a {\em self-stress} of $(G,\bp)$ as $\omega:E(G)\rightarrow \mathbb{R}$ satisfying
\begin{equation}\label{eq:equilibrium_p}
\sum_{j: ij\in E(G)} \omega(ij)(\bp(j)- \bp(i))^{p-1}=0\qquad (u\in V(G)),
\end{equation}
or, equivalently, $\omega$ is in the left kernel of $J(\pi_G\circ f_p^{\times d})(\bp)$.
(Recall our convention that 
$(\bp(j)- \bp(i))^{p-1}$ denotes the vector obtained by taking the $k$-th power of each coordinate of $\bp(j)-\bp(i)$.)
A  {\em $k$-th coordinated} self-stress is  $\omega^k:E(G)\rightarrow \mathbb{R}$  given by 
\begin{equation}\label{eq:coordinated}
\omega^k(ij):=\omega(ij) \left( (\bp(j))_k-(\bp(i)_k) \right)^{p-2}\qquad (ij\in E(G)).
\end{equation}
for some self-stress $\omega:E(G)\rightarrow \mathbb{R}$ of $(G,\bp)$.

Now we consider the Laplacian $L_{G,\omega^k}$ weighted by a $k$-th coordinated self-stress $\omega^k$.
Since $L_{G,\omega^k}$ is a Laplacian matrix, 
the all-one vector is in its kernel.
Also, since (\ref{eq:equilibrium_p}) and (\ref{eq:coordinated}) 
imply (\ref{eq:equilibrium}) in the $k$-th coordinate,
we have that $\rank L_{G,\omega^k}\leq n-2$.

We are now ready to state our main result.

\begin{theorem}\label{thm:main}
Let $p$ be an even positive integer with $p\neq 2$ and let $(G,\bp)$ be a generic 
 $2$-dimensional framework with $n \ge 3$ vertices.
Then $(G,\bp)$ is globally rigid in $\ell_p^2$ if and only if 
$(G,\bp)$ has a $k$-th coordinated self-stress $\omega^k$ for some $k$
such that $\rank L_{G,\omega^k}=n-2$.
\end{theorem}

Since the condition of Theorem~\ref{thm:main} is algebraic in the coordinates of $\bp$, Theorem~\ref{thm:main} implies that the global rigidity in the $\ell_p$-plane is a generic property of graphs. 
Combining this with Theorem~\ref{thm:DHN}, we have a Jackson and Jord{\'a}n type characterisation  in $\ell_p^2$, see Theorem~\ref{thm:last_main}.

We strongly believe that the same characterisation is valid in general.
\begin{conj}\label{conj:main}
Let $p$ be an even positive integer with $p\neq 2$, $d$ be a positive integer, and $(G,\bp)$ be a generic $d$-dimensional framework  with $n \ge 3$ vertices.
Then $(G,\bp)$ is globally rigid in $\ell_p^d$ if and only if 
$(G,\bp)$ has a $k$-th coordinated self-stress $\omega^k$ for some $k$
such that $\rank L_{G,\omega^k}=n-2$.
\end{conj}

It should be noted that the proof of the sufficiency is valid in any dimension. This is stated in Theorem~\ref{thm:COV_test} and it offers an $\ell_p$-analogue of Connelly's sufficient condition~\cite{connelly2005generic}.

\begin{eg}\label{ex: base case of 1-dim}
    As a simplest example, consider the case when $G = K_3$ and 
    $(G, \bx)$ is a 1-dimensional framework.
    Then,
    \begin{equation*}
    J(\pi_G \circ f_p^{\times 1})(\bx) = 
        \begin{blockarray}{cccc}
            \,& v_1 & v_2 & v_3\\
            \begin{block}{c(ccc)}
                v_1 v_2 & (\bx(v_1)-\bx(v_2))^{p-1} & (\bx(v_2)-\bx(v_1))^{p-1} & 0\\
                v_1 v_3 & (\bx(v_1)-\bx(v_3))^{p-1} & 0 & (\bx(v_3)-\bx(v_1))^{p-1}\\
                v_2 v_3 & 0 & (\bx(v_2)-\bx(v_3))^{p-1} & (\bx(v_3)-\bx(v_2))^{p-1}\\
            \end{block}
        \end{blockarray}.
    \end{equation*}
    It is straightforward to check that
    \begin{align*}
        \omega &= (\omega(v_1v_2), \omega(v_1v_3), \omega(v_2v_3))\\
        &= ((\bx(v_1) - \bx(v_2))^{1-p}, (\bx(v_3) - \bx(v_1))^{1-p}, (\bx(v_2) - \bx(v_3))^{1-p})
    \end{align*}
    is a unique self-stress of $(G, \bx)$ up to constant, and the Laplacian matrix weighted by the first-coordinated stress is
    \begin{equation*}
        L_{G, \omega^1} = \begin{pmatrix}
            \frac{1}{\bx(v_3) - \bx(v_1)} + \frac{1}{\bx(v_1) - \bx(v_2)} & -\frac{1}{\bx(v_1) - \bx(v_2)} & -\frac{1}{\bx(v_3) - \bx(v_1)} \\
            -\frac{1}{\bx(v_1) - \bx(v_2)} & \frac{1}{\bx(v_1) - \bx(v_2)} + \frac{1}{\bx(v_2) - \bx(v_3)} & -\frac{1}{\bx(v_2) - \bx(v_3)} \\
            -\frac{1}{\bx(v_3) - \bx(v_1)} & -\frac{1}{\bx(v_2) - \bx(v_3)} & \frac{1}{\bx(v_2) - \bx(v_3)} + \frac{1}{\bx(v_3) - \bx(v_1)}
        \end{pmatrix}.
    \end{equation*}
    $L_{G,\omega^1}$ has rank equal to $1 = |V(G)| - 2$. \qed
\end{eg}

The second of our main results is the following sufficient condition in terms of local rigidity in general dimension.
\begin{theorem}\label{thm:suff}
Let $p$ be an even positive integer with $p\neq 2$, $d$ be a positive integer, and $(G,\bp)$ be a generic $d$-dimensional framework  with $n$ vertices.
Then $(G,\bp)$ is globally rigid in $\ell_p^d$ if 
$G$ is 2-connected and $G$ is locally rigid in $\ell_p^{d+1}$.
\end{theorem}

An application of Theorem~\ref{thm:suff} is a sharp lower bound of the sampling probability for the Erd{\H o}s-R{\'e}nyi random graph to be globally rigid in $\ell_p^d$.
For non-negative integers $n, M$, the Erd{\H o}s-R{\'e}nyi random graph $G(n,M)$ is defined inductively by first setting $G(n, 0)$ to be the empty graph with $n$ vertices
and then setting $G(n,M)$ to be obtained from $G(n,M-1)$ by adding a new edge taken uniformly randomly from the complement of $G(n, M-1)$.
Consider the following three random  variables,
\begin{align*}
    M_{\ell_p^d-LR} &:= \min\{M: \text{$G(n, M)$ is locally rigid in $\ell_p^d$}\},\\
    M_{\ell_p^d-GR} &:= \min\{M: \text{$(G(n, M), \bp)$ is globally rigid in $\ell_p^d$ for every generic $\bp$}\}, \text{and}\\
    M_d &:= \min\{M: \delta(G(n, M)) = d\},
\end{align*}
where $\delta(H)$ denotes the minimum degree of a graph $H$.
Recently, Lew et al.~\cite{lew2023sharp} showed that 
$\mathrm{a.a.s.}$ $M_{\ell_2^d-LR} = M_d$ and $M_{\ell_2^d-GR} = M_{d+1}$ for Euclidean rigidity.
Their argument is still valid for local rigidity in $\ell_p$-space,
and hence Theorem~\ref{thm:suff} further implies the global rigidity threshold as follows.
\begin{corollary}
    Let $p$ be an even positive integer with $p \neq 2$ and $d$ be a positive integer.
    Then, $\mathrm{a.a.s.}$ $M_{\ell_p^d-LR} = M_d$ and $M_{\ell_p^d-GR} = M_{d+1}$.
\end{corollary}

The proofs of Theorems~\ref{thm:main} and~\ref{thm:suff} are given in Sections~\ref{sec:plane} and ~\ref{sec:suff}, respectively.

\subsection{Comments on the self-stress condition in Theorem~\ref{thm:main}}\label{subsec:remark}
Theorem~\ref{thm:main} (or Conjecture~\ref{conj:main}) can be used to develop an efficient method for checking if a/every generic realisation of $G$ is globally rigid.
This is an adaptation of a technique in Euclidean global rigidity~\cite{gortler2010characterizing}.


Let $(G, \bp)$ be a generic locally rigid framework in $\ell_p^d$, and $G_0$ be a minimally locally rigid spanning subgraph of $G$.
(Here a graph $H$ is said to be minimally locally rigid if $H$ is locally rigid and $H-e$ is not locally rigid for any $e\in E(H)$.)
For each $e \in E(G) \setminus E(G_0)$, there exists a unique self-stress $\omega_e$ of $(G_0+e,\bp)$ such that $\omega_e(e)=1$.
Observe that $\{\omega_e\}_{e \in E(G) \setminus E(G_0)}$ forms a basis of the linear space of self-stresses of $(G,\bp)$.
Hence, a self-stress $\omega$ of $(G,\bp)$ can be written, in a unique way, as
\begin{equation}\label{eq: linear combination of self-stress}
    \omega = \sum_{e \in E(G) \setminus E(G_0)} c_e \omega_e
\end{equation}
for some real numbers $\bc=\{c_e\}_{e\in E(G)\setminus E(G_0)}$.
We say that $\omega$ is \textit{generic} with respect to $G_0$ if  
$\{c_e\}_{e \in E(G) \setminus E(G_0)}$ is algebraically independent over the field $\mathbb{Q}(\bp)$ generated by $\mathbb{Q}$ and the entries of $\bp$,
and $\omega$ is {\em generic} if 
it is generic with respect to some minimally locally rigid spanning subgraph $G_0$ of $G$.

Since each entry of $\omega_e$ is a rational function in the entries of $\bp$, a self-stress is written as a rational map $\omega(\bp, \bc)$ of $\bp$ and $\bc$.
This implies that the entries in the Laplacian matrix $L_{G, \omega(\bp, \bc)^k}$ are also rational functions of $\bp$ and $\bc$.
Hence the rank of $L_{G, \omega(\bp, \bc)^k}$ will be maximised whenever both $\bp$ and $\omega$ are generic. 
Also, such rational functions in $\bp$  are symmetric with respect to permutations of coordinates.
Thus, we have the following proposition.
\begin{prop}($\ell_p$-analogue of a theorem of Connelly and Whiteley~\cite{connelly2005generic}) \label{prop: avoiding generic construction}
    Let $p$ be an integer with $p\geq 2$,
    and $(G,\bp)$ be a $d$-dimensional framework.
    If $(G,\bp)$ is infinitesimally rigid in $\ell_p^d$, then 
    $(G,\bq)$ is infinitesimally rigid in $\ell_p^d$ for every generic framework $(G,\bq)$.
    Moreover, if $(G,\bp)$ has  a self-stress $\omega$  with $\rank L_{G, \omega^k} = |V(G)|-2$ for some $k$, then $\rank L_{G, {\omega'}^{k'}} = |V(G)|-2$ holds for every generic framework $(G, \bq)$, every generic self-stress $\omega'$ of $(G,\bq)$, and every $k'\in \{1,\dots, d\}$.
\end{prop}
Proposition~\ref{prop: avoiding generic construction} says that, in order to confirm the stress condition in Theorem~\ref{thm:main} or Conjecture~\ref{conj:main}, it is sufficient to construct an infinitesimally rigid framework $(G,\bp)$ (which may not be generic) that admits  a self-stress $\omega$  with $\rank L_{G, \omega^k} = |V(G)|-2$ for some $k$.

\medskip

We also remark the following simple property of weighted Laplacians.
\begin{lemma}\label{lem:laplacian_property}
Let $G$ be a connected graph and $\omega:E(G)\rightarrow \mathbb{R}\setminus \{0\}$ be an edge weight.
Suppose $G$ has a $1$-separator $(H_1, H_2)$,
that is, $G=H_1\cup H_2$, $|V(H_1)\cap V(H_2)|=1$, and
$V(H_1)\setminus V(H_2)\neq \emptyset \neq V(H_2)\setminus V(H_1)$.
Let $\omega_i$ be the restriction of $\omega$ to $E(H_i)$.
\begin{itemize} 
\item[(i)] $\rank L_{G,\omega}=\rank L_{H_1,\omega_1}+\rank L_{H_2,\omega_2}$.
\item[(ii)] If $\omega$ is a self-stress of a $d$-dimensional framework $(G,\bp)$
and $\bp|H_i$ be the restriction of $\bp$ to $V(H_i)$ for $i=1,2$,
then $\omega_i$ is a self-stress of $(H_i,\bp|H_i)$.
\end{itemize}
\end{lemma}
\begin{proof}
Denote the vertex in $V(H_1)\cap V(H_2)$ by $v$,
and denote
\[
L_{H_i, \omega_i} = \begin{blockarray}{ccc}
        \, & V(H_i) \setminus \{v\} & v \\
        \begin{block}{c(cc)}
            V(H_i) \setminus \{v\}& A_i & b_i^{\top} \\
            v & b_i & c_i \\
        \end{block}
    \end{blockarray}
\]
Since $H_i$ is connected and each entry of $\omega_i$ is nonzero, $b_i$ is not a zero vector.
Moreover, since the sum of columns (resp., rows) of $L_{H_i,\omega_i}$ is zero, 
the rank of $A_i$ is equal to that of $L_{H_i,\omega_i}$.

By $G=H_1\cup H_2$, 
\begin{equation*}
    L_{G, \omega} = \begin{blockarray}{cccc}
        \,& V(H_1) \setminus \{v\} & v & V(H_2) \setminus \{v\}\\
        \begin{block}{c(ccc)}
           V(H_1) \setminus \{v\} & A_1 & b_1^{\top} & 0\\
           v & b_1 & c_1 + c_2 & b_2 \\
           V(H_2) \setminus \{v\} & 0 & b_2^{\top} & A_2\\
        \end{block}
    \end{blockarray}.
\end{equation*}
In this matrix,
adding all the columns of $V(H_1)\setminus \{v\}$ to that of $v$
and then adding all the rows of $V(H_1)\setminus \{v\}$ to that of $v$ 
result in 
\begin{equation}\label{eq:separator}
\begin{pmatrix}
        A_1 & 0 & 0 \\
        0 & c_2 & b_2 \\
        0 & b_2^{\top} & A_2
    \end{pmatrix}.
\end{equation}
Hence, 
$\rank L_{G, \omega} =
\rank A_1 +\rank L_{H_2,\omega_2}
=\rank L_{H_1,\omega_1}+\rank L_{H_2,\omega_2}.$
This completes the proof of (i).

To see (ii), recall the fact that 
infinitesimal rigidity in $\ell_p$-space is invariant under translations of a framework. This implies that 
the $d$ columns associated with vertex $v$ are redundant in the column space of $J(\pi_{H_i}\circ f_p^{\times d}(\bp|H_i))$,
and the equilibrium condition (\ref{eq:equilibrium_p}) at $v$ automatically holds if it holds at other vertices.
Hence, if $\omega$ is a self-stress of $(G,\bp)$, then $\omega_i$ is a self-stress of $(H_i,\bp|H_i)$.
\end{proof}

Lemma~\ref{lem:laplacian_property} implies that 
2-connectivity is necessary for the stress condition in Theorem~\ref{thm:main}.

\section{Identifiability and Global Rigidity}\label{sec:iden}
In this section we shall introduce tools from algebraic geometry,
which are used in the proofs of Theorems~\ref{thm:main} and~\ref{thm:suff}.
The materials in this section can be understood as a detailed exposition of \cite[Section 7.2]{cruickshank2023identifiability} by specializing discussion to $\ell_p$-rigidity.

Since our tools are taken from the context of projective varieties over the complex field, we shall look at the affine cone of projective varieties over the complex field.
It should be noted that our rigidity results are established over the real field and we look at frameworks with real point configurations when referring to rigidity.

Let ${\cal V}$ be an affine variety in $\mathbb{C}^m$.
The {\em $d$-secant} $S_d({\cal V})$ of ${\cal V}$ is defined as the Zariski closure of the union of linear subspaces spanned by $d$ points in ${\cal V}$.
If ${\cal V}$ is a cone, we have
\[
S_d({\cal V})=\overline{\left\{\sum_{k=1}^d x_k: x_1,\dots, x_d\in {\cal V}\right\}},
\]
where the upper bar denotes the Zariski closure.

For a point $x$ in ${\cal V}$, let $T_{x}{\cal V}$ be the tangent space of ${\cal V}$ at $x$.
If ${\cal V}$ is irreducible, a point $x$ in ${\cal V}$ is said to be {\em generic} if any polynomial with rational coefficients vanishing at $x$ is vanishing at every point in ${\cal V}$.

\subsection{Cayley-Menger varieties}\label{subsec:cayley-menger}
For a positive integer $n$, define 
$f_2: \mathbb{C}^{n}\rightarrow \mathbb{C}^{n\choose 2}$ by
\[
(f_2(\bx))_{ij}=(x_i-x_j)^2 \qquad (\bx=(x_1,\dots, x_n)\in \mathbb{C}^{n})
\]
for each $i,j$ with $1\leq i<j\leq n$.
When $\bx\in \mathbb{R}^{n}$,  $f_2(\bx)$ is the list of squared inter-point distances  over points $x_1, x_2, \dots, x_n$ in the real line.
Over $\mathbb{C}$, the Zariski closure of the image of $f_2$ is known as the {\em Cayley-Menger variety} ${\cal CM}_n$,
and the $d$-secant is called the $d$-dimensional Cayley-Menger variety.
The $d$-dimensional Cayley-Menger variety has been extensively studied in the context of rigidity theory, and as shown by Gortler, Healy, and Thurston~\cite{gortler2010characterizing} it plays a critical role in the analysis of generic global rigidity.

In this paper, we consider the $\ell_p$-analogue of the Cayley-Menger variety.
For a positive integer $p$, define 
$f_p: \mathbb{C}^{n}\rightarrow \mathbb{C}^{n\choose 2}$ by
\[
(f_p(\bx))_{ij}=(x_i-x_j)^p \qquad (\bx=(x_1,\dots, x_n)\in \mathbb{C}^{n})
\]
for each $i,j$ with $1\leq i<j\leq n$.
The Zariski closure of the image of $f_p$ is called the {\em $p$-Cayley-Menger variety} ${\cal CM}_n^p$.

Since $f_p^{\times d}$ is a polynomial map (defined in (\ref{eq:f_d})), it can be also defined over $\mathbb{C}$. 
Similarly, the projection map $\pi_G$ is defined over $\mathbb{C}$
by $\pi_G:\mathbb{C}^{n\choose 2}\rightarrow \mathbb{C}^{E(G)}$.
For $\bp\in (\mathbb{C}^d)^n$,
let $\bx_i\in \mathbb{C}^n$ be the $i$-th coordinate vector of $\bp$.
Then 
\[
\pi_G\circ f_p^{\times d}(\bp)=\sum_{i=1}^d \pi_G\circ f_p(\bx_i).
\]
This in particular implies that
$\overline{{\rm image}\ \pi_G\circ f_p^{\times d}}$ is the $d$-secant of $\overline{\pi_G({\cal CM}_n^p)}$.
(The above equation only implies 
$\overline{{\rm image}\ \pi_G\circ f_p^{\times d}}\subseteq S_d(\overline{\pi_G({\cal CM}_n^p)})$,
but both sets are irreducible and have the same dimension,
so they are equal.)

\medskip
\noindent
{\bf Remark.}
Although we have defined the $p$-Cayley-Menger variety for any positive integer $p$, its $d$-secant does not coincide with the image of the $\ell_p$-distance measurement map when $p$ is odd.
Indeed, if $p$ is odd, the one-dimensional $\ell_p$-distance measurement map with the $p$-th power becomes a piecewise polynomial map of the form $({\rm sign}(x_i-x_j))(x_i-x_j)^p$.
Accordingly, our connection between $\ell_p$-global rigidity and the identifiability  of the $p$-Cayley-Menger variety (given in Proposition~\ref{prop:iden_global} below) is currently limited to the case where $p$ is even.
Extending this connection to the case where $p$ is odd remains an open problem.
\subsection{Defectivity}\label{subsec:defective}
In this subsection, we shall first look at a relation between the local rigidity of graphs $G$ and the defectivity of $\overline{\pi_G({\cal CM}_n^p)}$.
The content of this section is described in more detail in \cite{cruickshank2023identifiability} in a general form,
and we will provide an exposition by restricting our discussion to ${\cal CM}_n^p$ for completeness.

Suppose  an affine variety $\mathcal{V}$ is a cone in $\mathbb{C}^m$.
Then the $d$-secant $S_d(\mathcal{V})$ has dimension at most $\min\{d \dim \mathcal{V}, m\}$.
The latter number is called the {\em expected dimension}, and $\mathcal{V}$ is said to be {\em $d$-defective} if $\dim S_d(\mathcal{V})$ is smaller than the expected dimension.

When ${\cal V}=\overline{\pi_G({\cal CM}_n^p)}$, 
the $d$-defectivity is directly related to local rigidity as follows.

\begin{prop}\label{prop:defective}
Let $p$ be an integer with $p>2$.
Suppose $G$ is a connected graph with $n$ vertices and $|E(G)|\geq dn-d$.
Then $\overline{\pi_G({\cal CM}_n^p)}$ is not $d$-defective if and only if  $G$ is locally rigid in $\ell_p^d$.
\end{prop}
\begin{proof}
Let ${\cal V}=\overline{\pi_G({\cal CM}_n^p)}$.
Then $\dim {\cal V}=n-1$ as long as $G$ is connected.
Hence, assuming that $G$ is connected and $|E(G)|\geq dn-d$, 
${\cal V}$ is $d$-defective if and only if $\dim S_d(\mathcal{V})<dn-d$.

Note that $\dim S_d(\mathcal{V})$ is equal to the dimension of the tangent space of $S_d(\mathcal{V})$ at a generic point.
Since $S_d(\mathcal{V})=\overline{{\rm image}\ \pi_G\circ f_p^{\times d}}$, for a generic $\bp$
the dimension of the tangent space of $S_d(\mathcal{V})$ at $f_p^{\times d}(\bp)$ is equal to 
$\rank J(\pi_G\circ f_p^{\times d})(\bp)$.
Therefore,  ${\cal V}$ is $d$-defective if and only if 
$\rank J(\pi\circ f_p^{\times d})(\bp)<dn-d$.

When $p$ is an integer with $p>2$, this condition is exactly the condition for the generic framework $(G,\bp)$ not to be locally rigid in $\ell_p^d$ by Proposition~\ref{prop:inf}. 
\end{proof}

\subsection{Identifiability}\label{subsec:iden}
Let ${\cal V}$ be an affine variety in $\mathbb{C}^m$ and suppose that 
${\cal V}$ is a cone.
${\cal V}$ is called {\em $d$-identifiable} if any generic point 
$x\in S_d({\cal V})$ can be written as 
$x=\sum_{k=1}^d x_k$ for unique $x_1,\dots, x_d\in {\cal V}$ up to a scaling of each $x_i$ and a permutation of indices.

When ${\cal V}$ is $\overline{\pi_G({\cal CM}_n^p)}$ and $\bp$ is generic, 
$\pi_G\circ f_p^{\times d}(\bp)$ is a generic point in the $d$-secant of 
$\overline{\pi_G({\cal CM}_n^p)}$.
Using this interpretation, we reduce the global rigidity problem to the $d$-identifiability problem of $\overline{\pi_G({\cal CM}_n^p)}$ in Proposition~\ref{prop:iden_global} below.

We first check that the $d$-identifiablity imposes the following connectivity condition.
\begin{prop}\label{prop:iden_connectivity}
Let $p$ be a positive integer, $d$ be an integer with $d\geq 2$, and $G$ be a graph with $n\geq 3$ vertices.
If $\pi_G({\cal CM}_n^p)$ is $d$-identifiable, then $G$ is 2-connected.
\end{prop}
\begin{proof}
Consider a generic $d$-dimensional framework $(G,\bp)$.
Then $\pi_G\circ f_p^{\times d}(\bp)$ is a generic point in the $d$-secant of 
$\overline{\pi_G({\cal CM}_n^p)}$.
If $G$ is not 2-connected, then by $n\geq 3$, $G$ can be written as $G=H_1\cup H_2$ such
that $E(H_1)\cap E(H_2)=\emptyset$ and $|V(H_1)\cap V(H_2)|\leq 1$.
We only consider the case when $|V(H_1)\cap V(H_2)|=1$ since the case when $|V(H_1)\cap V(H_2)|=0$ is easier.

Let $v$ be the vertex in $V(H_1)\cap V(H_2)$.
By translations, we may assume that $\bp(v)$ is at the origin.
We consider another framework $(G,\bq)$ obtained by permutating the coordinates of $\bp$ only on $V(H_2)\setminus V(H_1)$ (keeping the points of $V(H_1)$ fixed).
We have $\pi_G\circ f_p^{\times d}(\bp)=\pi_G\circ f_p^{\times d}(\bq)$,
while 
$\pi_G\circ f_p(\bp_i)$ and $\pi_G\circ f_p(\bq_i)$ are different for the coordinate vectors $\bp_i$ and $\bq_i$ by $d\geq 2$.
Hence, $\overline{\pi_G({\cal CM}_n^p)}$ is not $d$-identifiable.
\end{proof}

\begin{prop}\label{prop:iden_global}
Let $p$ be an even positive integer, $d$ be an integer with $d\geq 2$, and $(G,\bp)$ be a generic framework with $n\geq 3$ vertices.
If $\overline{\pi_G({\cal CM}_n^p)}$ is $d$-identifiable over $\mathbb{C}$, then $(G,\bp)$ is globally rigid in $\ell_p^d$.
\end{prop}
\begin{proof}
To see the global rigidity of $(G,\bp)$, 
pick any framework $(G,\bq)$  equivalent  to $(G,\bp)$ in $\ell_p^d$-space.
Let $\bp_i$ (resp.~$\bq_i$) be the $i$-th coordinate vector of $\bp$ (resp.~$\bq$) for each $i$.
Then 
\begin{equation}\label{eq:prop_iden_global}
\sum_{i=1}^d \pi_G\circ f_p(\bp_i)=\pi_G\circ f_p^{\times d}(\bp)=\pi_G\circ f_p^{\times d}(\bq)=\sum_{i=1}^d \pi_G\circ f_p(\bq_i).
\end{equation}
Since $\bp$ is generic, $\pi_G\circ f_p^{\times d}(\bp)$ is a generic point in the $d$-secant of $\overline{\pi_G({\cal CM}_n^p)}$.
Hence, by the $d$-identifiability of $\overline{\pi_G({\cal CM}_n^p)}$,
(\ref{eq:prop_iden_global}) implies that there exists a permutation $\sigma$ on $\{1,2\dots, d\}$ such that $\pi_G\circ f_p(\bq_i)=\lambda_i \left( \pi_G\circ f_p(\bp_{\sigma(i)})\right)$ for some $\lambda_i\in \mathbb{C}$ for each $i=1,\dots, d$. Without loss of generality, we may assume that the permutation $\sigma$ is the identity.
Then 
\begin{equation}\label{eq:prop33}
\pi_G\circ f_p(\bq_i)=\lambda_i \left( \pi_G\circ f_p(\bp_i)\right) \quad \text{for each $i=1,\dots, d$}.
\end{equation}

Suppose that $\lambda_i=1$ for all $i$.
Then $\pi_G\circ f_p(\bq_i)=\pi_G\circ f_p(\bp_i)$,
and hence $(G,\bp_i)$ and $(G,\bq_i)$ are equivalent as one-dimensional frameworks in $\ell_p^1$. Since $G$ is  2-connected by Proposition~\ref{prop:iden_connectivity}, one can apply Proposition~\ref{prop:1d} to show that $(G,\bq_i)$ is congruent to  $(G,\bp_i)$ for all $i$.
Hence, $(G,\bq)$ is congruent to $(G,\bp)$, and $(G,\bp)$ is globally rigid.

We now show that $\lambda_i=1$ for all $i$.
Since $\overline{\pi_G({\cal CM}_n^p)}$ is $d$-identifiable,
the minimum vertex degree of $G$ is at least $d$.
(Here we used the fact that, if the minimum degree of $G$ is less than $d$, then $(G,\bp)$ is not locally rigid~\cite{kitson2014infinitesimal} and hence a continuous motion of $(G,\bp)$ gives a certificate that $\pi_G({\cal CM}_n^p)$ is not $d$-identifiable.)
Pick any vertex $v$ of $G$ and let $u_1,\dots, u_d$ be distinct neighbours of $v$ in $G$. Let $d_{ij}$ be the $i$-th coordinate of $\bp(u_j)-\bp(v)$ for each $1\leq i, j\leq d$. 
Then, the $vu_j$-th entry of $\pi_G\circ f_p(\bp_i)$ is $d_{ij}^p$.
On the other hand, by (\ref{eq:prop33}),
 the $vu_j$-th entry of $\pi_G\circ f_p(\bq_i)$ is $\lambda_i d_{ij}^p$.
Substituting them into (\ref{eq:prop_iden_global}), we obtain
\begin{equation}\label{eq:d_ij}
\sum_{k=1}^d d_{ij}^p = \sum_{i=1}^d \lambda_i^p d_{ij}^p \qquad (j=1,\dots, d).
\end{equation}
Since $\bp$ is generic, the set of $d_{ij}$ over $1\leq i, j\leq d$ is algebraically independent over $\mathbb{Q}$,
and hence the $d\times d$ matrix $(d_{ij}^p)_{1\leq i,j\leq d}$ is non-singular. Thus (\ref{eq:d_ij}) implies $\lambda_i^p=1$ for each $i$.
Since the entries of $\bp$ and $\bq$ are real and $p$ is even, 
$\pi_G\circ f_p(\bq_i)$ and $\pi_G\circ f_p(\bp_i)$ are both non-negative real numbers. Hence, by (\ref{eq:prop33}) and $\lambda_i^p=1$, $\lambda_i=1$ follows.
This completes the proof.
\end{proof}

\subsection{Tangential weak defectivity}\label{subsec:twd}
The identifiability problem of varieties has been extensively studied in algebraic geometry.
In this paper, we shall use the idea of tangential weak defectivity introduced by Chiantini and Ottaviani~\cite{chiantini2012generic} to analyse the identifiability of $\overline{\pi_G({\cal CM}_n^p)}$.

Let ${\cal V}\subseteq \mathbb{C}^m$ be the affine cone of a projective variety and $x_1,\dots, x_d$ be generic points in ${\cal V}$.
The {\em $d$-tangent contact locus} at $x_1,\dots, x_d$ is defined by
\[
C_{x_1,\dots, x_d}{\cal V}:=\overline{\{y\in {\cal V}: \text{$y$ is non-singular, } T_y{\cal V} \subset \langle T_{x_1} {\cal V}, \dots, T_{x_d} {\cal V}\rangle \}}.
\]
Since $x_1,\dots, x_d$ are generic, 
the Terracini lemma implies $T_{\sum_{i=1}^d x_i} S_d({\cal V})= \langle T_{x_1}{\cal V},\dots, T_{x_d}{\cal V}\rangle$.
Hence the $d$-tangent contact locus is equivalently defined as
\[
C_{x_1,\dots, x_d}{\cal V}:=\overline{\{y\in {\cal V}: \text{$y$ is non-singular, } T_y{\cal V} \subset T_{\sum_{i=1}^d x_i} S_d({\cal V}) \}},
\]
Since ${\cal V}$ is a cone, $\dim C_{x_1,\dots, x_d}{\cal V}\geq 1$.
${\cal V}$ is  {\em $d$-tangentially weakly defective} if 
an irreducible component of $C_{x_1,\dots, x_d}{\cal V}$ that contains $x_i$ has dimension at least two for some $i$.

Chiantini and Ottaviani~\cite{chiantini2012generic} showed that, if  ${\cal V}$ is not $d$-tangentially weakly defective, then ${\cal V}$ is $d$-identifiable.
In terms of the $p$-Cayley-Menger varieties, we have the following.
\begin{prop}\label{prop:twnd_iden}
Let $p$ be an integer with $p>2$ and $G$ be a connected graph with $n \ge 3$ vertices.
If $\overline{\pi_G({\cal CM}_n^p)}$ is not $d$-tangentially weakly defective,
then  $\overline{\pi_G({\cal CM}_n^p)}$ is $d$-identifiable.
\end{prop}
Since Proposition~\ref{prop:twnd_iden} is a key fact in this paper and it seems a useful tool in other rigidity problems, we are going to give a proof in Section~\ref{subsec:proof of twnd_iden}. 

In view of Proposition~\ref{prop:twnd_iden}, the next question we should address is how we can check the $d$-tangential weak non-defectivity. 
Bocci, Chiantini, Ottaviani, and Vannieuwenhoven~\cite{BCC,chiantini2014algorithm,chiantini2017generic} gave a linear algebraic technique to test this property for Segre varieties.
Interestingly, an adaptation of their idea to the $p$-Cayley-Menger varieties leads to our graph Laplacian condition in Section~\ref{subsec:main}.
The details are explained in the next lemma.
\begin{lemma}\label{lem:COV_test}
Let $p$ be an integer with $p>2$ and $(G,\bp)$ be a generic $d$-dimensional framework with $n \ge 2$ vertices.
If $(G,\bp)$ has a $k$-th coordinated self-stress $\omega^k$ such that 
$\rank L_{G,\omega^k}=n-2$ for all $k\in\{1,\dots, d\}$, then $\overline{\pi_G({\cal CM}_n^p)}$ is not $d$-tangentially weakly defective.
\end{lemma}
\begin{proof}
For simplicity we use the following notations.
For $i=1,\dots, d$, let $\bp_i\in \mathbb{R}^n$ be the $i$-th coordinate vector of $\bp$,
and let $x_i=\pi_G \circ f_p(\bp_i)$. 
Let ${\cal V}=\overline{\pi_G({\cal CM}_n^p)}$,
and ${\cal C}$ be the $d$-tangent contact locus of ${\cal V}$ at $x_1,\dots, x_d$.
Let ${\cal C}'$ be an irreducible component of ${\cal C}$ that contains $x_k$.
Instead of proving $\dim {\cal C}'=1$ directly, we shall look at $(\pi_G \circ f_p)^{-1}({\cal C}'\cap {\cal V})$.
Since $\pi_G \circ f_p$ is invariant by the addition of a constant to each coordinate
and $\pi_G\circ f_p$ is locally surjective around $\bp$ onto the image, $(\pi_G \circ f_p)^{-1}({\cal C}'\cap {\cal V})$ has at least one larger dimension than ${\cal C}'$.
So in order to prove $\dim {\cal C}'=1$,
it is sufficient to prove that $(\pi_G \circ f_p)^{-1}({\cal C}'\cap {\cal V})$ is contained in a two-dimensional variety.

Now, by the lemma assumption $(G,\bp)$ has a $k$-th coordinated self-stress 
$\omega^k$ such that $\rank L_{G,\omega^k}=n-2$.
By the definition of $k$-th coordinated self-stresses, there is a self-stress $\omega$ of $(G,\bp)$ from which $\omega^k$ is constructed via (\ref{eq:coordinated}).
Since the equilibrium condition  (\ref{eq:equilibrium_p}) for $\omega$ is equivalent to 
being in the left kernel of $J(\pi_G\circ f_p^{\times d})(\bp)$, 
the fact that $\omega$ is a self-stress of $(G,\bp)$
implies that $\omega$ is in the left kernel of $J(\pi_G\circ f_p^{\times d})(\bp)$ as a vector in $\mathbb{C}^{E(G)}$.
Moreover, the column space of $J(\pi_G\circ f_p^{\times d})(\bp)$ is 
$\langle T_{x_1}{\cal V}, T_{x_2}{\cal V}, \dots, T_{x_d}{\cal V}\rangle$,
so we conclude that $\omega$ is 
in the orthogonal complement of 
$\langle T_{x_1}{\cal V}, T_{x_2}{\cal V}, \dots, T_{x_d}{\cal V}\rangle$.

For each $\by \in (\pi_G \circ f_p)^{-1}({\cal C}'\cap {\cal V})$, 
we have 
$T_{\pi_G \circ f_p(\by)}{\cal V}\subseteq \langle T_{x_1}{\cal V}, T_{x_2}{\cal V}, \dots, T_{x_d}{\cal V}\rangle$ by the definition of ${\cal C}$,
and thus $\omega$ is in the orthogonal complement of $T_{\pi_G \circ f_p(\by)}{\cal V}$. The latter condition is equivalent to saying that 
$\omega$ is a self-stress of the one-dimensional framework $(G,\by)$.
Hence, we obtain
\[
(\pi_G \circ f_p)^{-1}({\cal C}'\cap {\cal V})\subseteq {\cal S}_{\omega}:=\{\by\in \mathbb{C}^n: \text{$\omega$ is a self-stress of $(G,\by)$}\}. 
\]

By the definition of self-stresses, 
$\by\in {\cal S}_{\omega}$ if and only if 
\[
\sum_{v:uv\in E(G)} \omega(uv)(\by(v)-\by(u))^{p-1}=0 \qquad (u\in V(G)).
\]
This is a system of $n$ polynomial equations with $n$ variables $\by(1), \by(2), \dots, \by(n)$.
Let $h_{\omega}:\mathbb{C}^n\rightarrow \mathbb{C}^n$ be the polynomial map representing this system of polynomial equations.
Then $\by\in {\cal S}_{\omega}$ if and only if $h_{\omega}(\by)=0$.
So ${\cal S}_{\omega}$ is an algebraic set.
Observe that $Jh_{\omega}(\by)$ is exactly the Laplacian matrix $L_{G,\omega'}$ of $G$ weighted by $\omega'$ given by 
\[
\omega'(uv):=(p-1)\cdot \omega(uv)\cdot (\by(v)-\by(u))^{p-2}\qquad (uv\in E(G)).
\]
Hence, $Jh_{\omega}(\bp_k)=(p-1)L_{G,\omega^k}$.
By the lemma assumption, $\rank L_{G,\omega^k}=n-2$.
So $\dim \ker Jh_{\omega}(\bp_k)=2$.
Since $\dim \ker Jh_{\omega}(\by)\geq 2$ for any $\by\in {\cal S}_{\omega}$,
we conclude that 
$\bp_k$ is a non-singular point in ${\cal S}_{\omega}$,
and the irreducible component of ${\cal S}_{\omega}$ that contains $\bp_k$ is two-dimensional.
This completes the proof.
\end{proof}

Combining the results so far, we have the following sufficient condition for global rigidity in $\ell_p^d$.
\begin{theorem}\label{thm:COV_test}
Let $p$ be an even positive integer with $p\neq 2$, $d$ be a positive integer, and $(G,\bp)$ be a generic framework in $\ell_p^d$ with $n \ge 3$ vertices.
If  $(G,\bp)$ has a $k$-th coordinated self-stress $\omega^k$ such that $\rank L_{G,\omega^k}=n-2$ for all $k$, then 
$(G,\bp)$ is globally rigid in $\ell_p^d$.
\end{theorem}
\begin{proof}
Suppose the stress condition holds.
By Lemma~\ref{lem:laplacian_property}, $G$ is 2-connected.
Also, by Lemma~\ref{lem:COV_test}, ${\cal V}=\overline{\pi_G({\cal CM}_n^p)}$ is not $d$-tangentially weakly defective.
By Proposition~\ref{prop:twnd_iden}, ${\cal V}$ is $d$-identifiable.
Finally, by the 2-connectivity of $G$, $n\geq 3$, and the $d$-identifiability, 
the global rigidity of $(G,\bp)$ follows from Proposition~\ref{prop:iden_global}.
\end{proof}


\subsection{Proof of Proposition~\ref{prop:twnd_iden}}\label{subsec:proof of twnd_iden}
Our proof of Proposition~\ref{prop:twnd_iden} is an adaptation of that of Chiantini, Ottaviani, and Vannieuwenhoven~\cite{chiantini2014algorithm} for Segre varieties.
We need the following lemma.
\begin{lemma}\label{lem:trisecant}
Let $p$ be an integer with $p> 2$, $G=(V,E)$ be a graph, 
and $x_1, x_2,\dots, x_d$ be generic points in $\overline{\pi_G({\cal CM}_n^p)}\subset \mathbb{C}^{E(G)}$.
If $G$ is connected with $|E(G)|\geq d(n-1)+1$, then 
$\overline{\pi_G({\cal CM}_n^p)}\cap \langle x_1,\dots, x_d\rangle=\langle x_1\rangle \cup  \langle x_2\rangle \cup \dots \cup \langle x_d\rangle$.
\end{lemma}
\begin{proof}
This is a direct consequence of a generalised trisecant lemma~\cite{chiantini2002weakly}, and the proof is identical to that of Proposition 7.13 in \cite{cruickshank2023identifiability}.
\end{proof}

\begin{lemma}\label{lem:number_of_edges}
Let $p$ be an integer with $p>2$,
and $(G,\bp)$ be a $d$-dimensional generic framework with a connected graph $G$ with $|E(G)| \ge 2$.
If $\overline{\pi_G({\cal CM}_n^p)}$ is not $d$-tangentially weakly defective, then 
\begin{itemize}
\item[(i)] $(G,\bp)$ is not independent (i.e., $J(\pi_G\circ f_p^{\times d})(\bp)$ is row dependent),
\item[(ii)] $(G,\bp)$ is locally rigid in $\ell_p^d$, and
\item[(iii)] $|E(G)|\geq d(n-1)+1$.
\end{itemize}
\end{lemma}
\begin{proof}
Let ${\cal V}=\overline{\pi_G({\cal CM}_n^p)}$.
To see (i), recall that  the column space of $J(\pi_G\circ f_p^{\times d})(\bp)$ is equal to the 
the tangent space of the $d$-secant of ${\cal V}$ at $\pi_G\circ f_p^{\times d}(\bp)$.
Hence, if $J(\pi_G\circ f_p^{\times d})(\bp)$ is row independent, the $d$-tangent contact locus would be the whole variety, whose dimension is $|E(G)|\geq 2$. This contradicts the assumption that ${\cal V}$ is not $d$-tangentially weakly defective.

For the proof of (ii),
we use an elementary fact that $J(\pi_G\circ f_p^{\times d})(\bp)$ is  row independent if $|E(G)|\leq d$, see, e.g.~\cite{dewarkitson2022}.
Hence, by (i), we have $|E(G)|> d$.
Suppose to the contrary that $(G,\bp)$ is not locally rigid. There is a continuous motion $\bp^t$ parameterised by $t\in [0,1]$.
Then by definition $\pi_G\circ f_p^{\times d}(\bp)=\pi_G\circ f_{p}^{\times d}(\bp^t)$, $\bp^0=\bp$, and $\bp^t$ is not congruent to $\bp$ for $t\in (0,1]$.

Let $\bp^t_k$ be the $k$-th coordinate vector of $\bp^t$.
By $\pi_G\circ f_p^{\times d}(\bp)=\pi_G\circ f_{p}^{\times d}(\bp^t)$,
the $d$-tangent contact locus of ${\cal V}$ at $\pi_G\circ f_p(\bp_1^t),\dots,  \pi_G\circ f_p(\bp_d^t)$ is invariant over $t$.
This in particular implies 
\begin{equation}\label{eq:number_of_edges}
\pi_G\circ f_p(\bp_k^t) \in C_{\pi_G\circ f_p(\bp_1),\dots, \pi_G\circ f_p(\bp_d)} {\cal V}
\end{equation}
for each $k=1,\dots, d$.
We now show that there is some $k$ such that 
$\pi_G\circ f_p(\bp_k^t)$ is not a scalar multiple of $\pi_G\circ f_p(\bp_k)$ for all small $t$.
To see this, first observe that  
$\pi_G\circ f_p(\bp_1),\dots, \pi_G\circ f_p(\bp_d)$ are linearly independent as long as $|E(G)|\geq d$.
Hence, if $\pi_G\circ f_p(\bp_k^t)$ is a scalar multiple of $\pi_G\circ f_p(\bp_k)$ for all $k$,
then $\pi_G\circ f_p^{\times d}(\bp)=\pi_G\circ f_p^{\times d}(\bp^t)$ would imply 
$\pi_G\circ f_p(\bp_k^t)=\pi_G\circ f_p(\bp_k)$ for all $k$.
On the other hand, the connectivity of $G$ implies that the one-dimensional framework $(G,\bp^t_k)$ is locally rigid,
and hence $\pi_G\circ f_p(\bp_k^t)=\pi_G\circ f_p(\bp_k)$ further implies that $\bp^t_k$ is a translation image of $\bp_k$
for all $k$ and $\bp^t$ is congruent to $\bp$, a contradiction.

Thus, there is some $k$ such that 
$\pi_G\circ f_p(\bp_k^t)$ is not a scalar multiple of $\pi_G\circ f_p(\bp_k)$ for all small $t$.
By (\ref{eq:number_of_edges}), 
$\pi_G\circ f_p(\bp_k^t)$ forms a curve in $C_{\pi_G\circ f_p(\bp_1),\dots, \pi_G\circ f_p(\bp_d)} {\cal V}$, contradicting the $d$-tangential weak non-defectivity of  ${\cal V}$.
This completes the proof of (ii).

Finally, (iii) follows from (i) and (ii).
Indeed, (i) and (ii) imply that 
$\rank J(\pi_G\circ f_p^{\times d})(\bp)\geq d(n-1)$
and $J(\pi_G\circ f_p^{\times d})(\bp)$ is not row independent.
So, $|E(G)|\geq 1+\rank J(\pi_G\circ f_p^{\times d})(\bp)
\geq d(n-1)+1$.
\end{proof}

\begin{proof}[Proof of Proposition~\ref{prop:twnd_iden}]
For simplicity, denote ${\cal V}=\pi_G({\cal CM}_n^p)$.
For $k=1,\dots, d$, let $\bp_k$ be the $k$-th coordinate vector of $\bp$,
and let $x_k=\pi_G \circ f_p(\bp_k)$.
Then $z:=\pi_G\circ f_p^{\times d}(\bp)=\sum_{k=1}^d x_k$.
To show the identifiability of ${\cal V}$,
we assume $z$ has another representation 
$z=\sum_{k=1}^d y_k$ for some $y_1,\dots, y_d\in {\cal V}$ with $y_1\notin \langle x_1\rangle \cup \dots \cup \langle x_d\rangle$.

By Lemma~\ref{lem:number_of_edges},
$|E(G)|\geq d(n-1)+1$, and hence 
Lemma~\ref{lem:trisecant} implies that 
${\cal V}\cap \langle x_1,\dots, x_d\rangle=\langle x_1\rangle\cup \dots \cup \langle x_d\rangle$.
In particular, 
$y_1\notin \langle x_1,\dots, x_d\rangle$.

Let $\lambda(t):[0,1]\rightarrow \mathbb{R}$ be a continuously increasing function with $\lambda(0)=1$,
and define $z^t=\lambda(t)y_1+\sum_{i=2}^d y_i$.
Since ${\cal V}$ is a cone, $z^t$ is still in $S_d({\cal V})$, and $T_z S_d({\cal V})=T_{z^t} S_d({\cal V})$ for all $t$.
Also, since $y_1 \notin \langle x_1,\dots, x_d\rangle$,
$z^t\notin \langle x_1,\dots, x_d\rangle$ for all $t>0$.

Since $f_p^{\times d}$ is a submersion in a neighbourhood of $\bp$, there is a continuous path $\bp^t$ starting from $\bp$ such that $f_p^{\times d}(\bp^t)=z^t$ for all sufficient small $t$.
Let $\bp^t_k$ be the $k$-th coordinate vector of $\bp^t$,
and let $x_k^t=f_p(\bp_k^t)$.
By $T_z S_d({\cal V})=T_{z^t} S_d({\cal V})$,
$T_{x_k^t} {\cal V}\subset T_z S_d({\cal V})$,
and hence $x_k^t\in C_{x_1,\dots, x_d}{\cal V}$.
Moreover, by $z^t\notin \langle x_1,\dots, x_d\rangle$ for $t>0$, there must be some $k$ such that 
$x_k^t$ is not a scalar multiple of $x_k$ for all small $t$.
This $x^t_k$ forms a  smooth path in $C_{x_1,\dots, x_d}{\cal V}$ starting from $x_k$, contradicting the assumption that ${\cal V}$ is not $d$-tangentially weakly defective. This completes the proof.
\end{proof}

\section{Sufficient condition in terms of local rigidity}\label{sec:suff}
In this section, we give a proof of Theorem~\ref{thm:suff}, a sufficient condition for global rigidity in $\ell_p^d$.
Our main tool is the following sufficient condition for $d$-identifiablility 
by Massarenti and Mella~\cite{massarenti2023bronowski}.
\begin{theorem}[\cite{massarenti2023bronowski}]\label{thm:Massarenti and Mella}
    Let $\mathcal{V}$ be the affine cone of a non-degenerate projective variety in $\mathbb{C}^m$. Suppose $(d+1)\dim \mathcal{V} \le m$ and $\mathcal{V}$ is neither $(d+1)$-defective nor $1$-tangentially weakly defective.
    Then $\mathcal{V}$ is $d$-identifiable.
\end{theorem}
In terms of the $p$-Cayley-Menger varieties, we have the following.
\begin{corollary}\label{cor:MM}
Let $p$ be an integer with $p >2$, $d$ be a positive integer, and $G$ be a graph.
If $G$ is locally rigid in $\ell_p^{d+1}$ and $\overline{\pi_G({\cal CM}_n^p)}$ is not $1$-tangentially weakly defective,
then $\overline{\pi_G({\cal CM}_n^p)}$ is $d$-identifiable.
\end{corollary}
\begin{proof}
Let ${\cal V}=\overline{\pi_G({\cal CM}_n^p)}\subseteq \mathbb{C}^m$.
Since $G$ is locally rigid in $\ell_p^{d+1}$, 
${\cal V}$ is not $(d+1)$-defective by Proposition~\ref{prop:defective}.
The local rigidity in $\ell_p^{d+1}$ also implies that $G$ is connected and $m\geq (d+1)n-(d+1)$.
The connectivity of $G$ implies the local rigidity of $G$ in $\ell_p^1$, and hence $\dim {\cal V}=n-1$. So $m\geq (d+1)\dim {\cal V}$ follows.
Also ${\cal V}$ is not 1-tangentially weakly defective by the assumption in the statement. 
Thus, 
${\cal V}$ is $d$-identifiable by Theorem~\ref{thm:Massarenti and Mella}.
\end{proof}

In order to apply Corollary~\ref{cor:MM} to the global rigidity problem in $\ell_p^d$,
what is missing is the analysis of the 1-tangential weak defectivity of $\overline{\pi_G({\cal CM}_n^p)}$.
In view of Lemma~\ref{lem:COV_test}, this is a question about one-dimensional frameworks. Specifically,  Lemma~\ref{lem:COV_test} states that
$\overline{\pi_G({\cal CM}_n^p)}$ is not 1-tangentially weakly defective if 
a generic 1-dimensional framework $(G,\bx)$ has a coordinated self-stress $\omega^1$ such that $\rank L_{G,\omega^1}=n-2$.
We now check this by an inductive argument as follows.

We need the following elementary tool from linear algebra.
Suppose $M = \begin{pmatrix}
    A & B \\
    C & D
\end{pmatrix}$ is a block matrix and $A$ is invertible.
Then, the \textit{Schur complement} of $A$ in $M$ is the matrix $M/A := D - C A^{-1} B$. We have
\begin{equation}\label{eq:Schur complement}
    \rank M = \rank A + \rank (M/A).
\end{equation}

\begin{lemma}\label{lem: edge subdivision}
    Let $(G, \bx)$  be a generic $d$-dimensional framework.
    Suppose $G$ is 2-connected and a graph $G'$ is obtained from $G$ by subdividing an edge $e$.
    Suppose also that $(G, \bx)$ has a generic self-stress $\omega$ such that $\omega(e)\neq 0$ and $\rank L_{G, \omega^1} = |V(G)| - 2$.
    Then, there is a generic one-dimensional framework $(G', \bx')$ which has a self-stress $\omega'$ such that $\rank L_{G, {\omega'}^1} = |V(G')| - 2$.
\end{lemma}
\begin{proof}
    Suppose that $G'$ is constructed from $G$ by deleting $e=v_1 v_2$, adding a vertex $v_0$ and two new edges $v_0v_1$ and $v_0v_2$.
    By Proposition \ref{prop: avoiding generic construction}, it suffices to show that there exist a framework $(G', \bx')$ with $\rank  J(\pi_{G'}\circ f_p^{\times 1})(\bx') = |V(G')| - 1$ and a self-stress $\omega'$ such that $\rank L_{G', {\omega'}^{1}} = |V(G')| - 2$.
    Choose a one-dimensional point configuration $\bx':V(G')\rightarrow \mathbb{R}$ such that $\bx'(v_0)=(\bx(v_1)+\bx(v_2))/2$ and 
    $\bx'(u)=\bx(u)$ for $u\in V(G)$.
    We first check $\rank J(\pi_{G'}\circ f_p^{\times 1})(\bx')=|V(G')|-1$.
    
    To see this, consider the rows of $J(\pi_{G'}\circ f_p^{\times 1})(\bx')$ corresponding to the two new edges $v_0v_1, v_0v_2$:
    \begin{equation*}
        \begin{blockarray}{ccccc}
            \,& v_0 & v_1 & v_2 & V(G)\setminus\{v_1,v_2\}\\
            \begin{block}{c(cccc)}
                v_0v_1 & -2^{1-p}d & 2^{1-p}d & 0 & 0\\
                v_0v_2 & 2^{1-p}d & 0 & -2^{1-p}d & 0\\
            \end{block}
        \end{blockarray},
    \end{equation*}
    where $d := (\bx(v_1) - \bx(v_2))^{p-1}$.
    By $\bx'(v_0)=(\bx'(v_1)+\bx'(v_2))/2$,
    we can use elementary row operations to reduce $J(\pi_{G'}\circ f_p^{\times 1})(\bx')$ to the form
    \begin{equation*}
        \left(
        \begin{array}{c | c}
        d & {\begin{array}{cc}-d \ & \  0 \end{array}}\\
        \hline
        O & J(\pi_{G}\circ f_p^{\times 1})(\bx)
        \end{array}
        \right).
    \end{equation*}
    Thus, $\rank J(\pi_{G'}\circ f_p^{\times 1})(\bx') = \rank J(\pi_{G}\circ f_p^{\times 1})(\bx) + 1 = |V(G')| - 1$.

    Next, using the self-stress $\omega$ of $(G,\bx)$,  we define a self-stress $\omega'$ of $(G',\bx')$ as follows:
    \begin{equation*}
        \omega'(e) = \begin{cases}
            2^{p-1}\omega(v_1v_2) & (e \in \{v_0v_1, v_0v_2\})\\
            \omega(e) & (e \in E(G)\setminus \{v_1v_2\}),
        \end{cases}
    \end{equation*}
    for $e\in E(G)$.
    It is easy to see that $\omega' \in \ker J(\pi_{G'}\circ f_p^{\times 1})(\bx')^{\top}$, and hence $\omega'$ is a self-stress of $(G', \bx')$.
    Then, the Laplacian matrix $L_{G', {\omega'}^1}$ is given by
    \begin{equation*}
        L_{G', {\omega'}^1} = \left(\begin{array}{c|c}
            0 & 0 \\
            \hline
            0 & L_{G, \omega^1}\\
        \end{array}\right)
        + 
        \left(\begin{array}{c | c}
        \begin{array}{ccc}
        4\omega(v_1v_2) & -2\omega(v_1v_2) & -2\omega(v_1v_2)\\
        -2\omega(v_1v_2) & \omega(v_1v_2) & \omega(v_1v_2)\\
        -2\omega(v_1v_2) & \omega(v_1v_2) & \omega(v_1v_2)\\
        \end{array}
        & O \\
       \hline
       O & O
        \end{array}
        \right).
    \end{equation*}
    Since  $\omega(v_1v_2) \neq 0$, we can take the Schur complement of $L_{G', {\omega'}^1}$ at the left-top corner.
    The resulting matrix $L_{G', {\omega'}^1}/v_0$ is exactly $L_{G, \omega^1}$.
    Thus, $\rank L_{G', {\omega'}^1} = 1 + \rank L_{G', {\omega'}^1}/v_0 = 1 + \rank L_{G, \omega^1} = |V(G')| - 2$.
\end{proof}

In a graph $G$, an inclusion-wise maximal 2-connected subgraph (having at least three vertices) is called a {\em 2-connected component}
and an inclusion-wise maximal 2-edge-connected subgraph (which may consist of a single vertex) is called a {\em 2-edge-connected component}. 

\begin{theorem}\label{thm:laplacian of 1-dim}
    Let $p$ be an integer with $p> 2$ and $(G, \bx)$ be a generic one-dimensional framework.
    Then $(G, \bx)$ has a $1$-st coordinated self-stress $\omega^1$ such that $\rank L_{G,\omega^1} = |V(G)|-a-b$,
    where $a$ and $b$ denote the number of 2-connected components and that of 2-edge-connected components in $G$, respectively.
\end{theorem}
\begin{proof}
First suppose that $G$ is 2-connected.
Then $a=b=1$.
It is a basic fact from graph theory that 
$G$ can be constructed inductively from $K_3$ by subdividing edges and adding edges keeping 2-connectivity.
By Proposition~\ref{prop:inf}, $G$ and the intermediate graphs are redundantly locally rigid in $\ell_p^1$.
We have already seen in Example~\ref{ex: base case of 1-dim}
that $K_3$ satisfies the property of the statement. 
Also, by the redundant rigidity of $G$ and Lemma~\ref{lem: edge subdivision}, for every intermediate graph $H$ and its generic point configuration, there exists a generic self-stress $\omega_H$ such that $\rank L_{H, \omega_H^1} = |V(H)| - 2=|V(H)|-a-b$.

Suppose secondly that $G$ is not 2-connected but 2-edge-connected with at least two vertices.
Then  the set of all 2-connected components form a tree structure  (known as the block-cut-tree).
Applying Lemma~\ref{lem:laplacian_property} along this tree structure, 
we obtain $\rank L_{G,\omega^1}=|V(G)|-a-1$ for any generic framework $(G,\bx)$ and any generic stress $\omega$ of $(G,\bx)$.

Finally, we consider a general case.
Let $\omega$ be a generic self-stress of $(G,\bx)$.
For any bridge $e$ in $G$, $(G-e,\bx)$ is not rigid 
and hence $\omega(e)=0$ holds.
Therefore, if we denote by $H$ the graph obtained from $G$ by removing all bridges of $G$, then 
the restriction $\omega'$ of $\omega$ to $E(H)$ is a self-stress of $(H,\bx)$.
Since $\omega(e)=0$ for every bridge $e$, we have 
$L_{G,\omega^1}=L_{H,\omega'^1}$.
So $\rank L_{G,\omega^1}$ is equal to the sum of the ranks of the Laplacian sub-matrices over all connected components in $H$. 
Note that the set  of connected components in $H$ coincides with the set of 2-edge-connected components in $G$ (some of which may consist of  a single vertex). 
If a 2-edge-connected component $C$ has at least two vertices, then  the weighted Laplacian has rank $|V(C)|-a_C-1$, where $a_C$ denotes the number of 2-connected components in $C$.
If a 2-edge-connected component $C$ consists of a single vertex,
then the corresponding weighted Laplacian has rank zero,
which is $|V(C)|-a_C-1$ (by $a_C=0$).
Taking the sum of those values over all 2-edge-connected components, we obtain the desired rank formula.
\end{proof}


We are now ready to complete the proof of Theorem~\ref{thm:suff}.
\begin{proof}[Proof of Theorem~\ref{thm:suff}]
Let ${\cal V}=\overline{\pi_G({\cal CM}_n^p)}$.
By the 2-connectivity of $G$ and Theorem~\ref{thm:laplacian of 1-dim},
a generic 1-dimensional framework $(G,x)$ has a $1$-st coordinated self-stress $\omega^1$ such that $\rank L_{G, \omega^1} = n-2$.
Hence, by Lemma~\ref{lem:COV_test}, ${\cal V}$ is not 1-tangentially weakly defective.
By Corollary~\ref{cor:MM}, ${\cal V}$ is $d$-identifiable.
By Proposition~\ref{prop:iden_global}, 
any generic framework $(G,\bp)$ is globally rigid in $\ell_p^d$.
\end{proof}

\section{Characterisation in the plane}\label{sec:plane}
In this section we give a proof of Theorem~\ref{thm:main}, a characterisation of generic global rigidity in the $\ell_p$-plane. 
In view of Theorem~\ref{thm:COV_test} and Theorem~\ref{thm:necessity}, what remains for Theorem~\ref{thm:main} is to prove the following sufficiency part.
\begin{theorem}\label{thm:sugiyama}
Let $p$ be an integer with $p>2$ and $(G,\bp)$ be a generic two-dimensional framework with $n$ vertices.
If $G$ is redundantly 2-tree-connected, then $(G,\bp)$ has a self-stress $\omega$ such that $\rank L_{G,\omega^k}=n-a-1$ for all $k \in \{1,2\}$, where $a$ denotes the number of 2-connected components in $G$.
\end{theorem}

The proof of Theorem~\ref{thm:sugiyama} follows the same story as 
the argument of Connelly~\cite{connelly2005generic} and Jackson and Jord{\'a}n~\cite{jackson2005connected} for characterising global rigidity in the Euclidean plane, and it  consists of the following three steps:
(Step 1) Prove that a graph in the target graph class can be constructed from small graphs by a sequence of simple graph operations;
(Step 2) Prove that the base graphs in the above construction satisfy Theorem~\ref{thm:sugiyama};
(Step 3) Prove that each graph operation preserves the property in the statement of Theorem~\ref{thm:sugiyama}.
Step~1 has been already studied extensively~\cite{jackson2019,dewar2024uniquely}, so we can reuse those results. 
We consider the following two graph operations.

    Let $G = (V, E)$ be a simple graph.
    A graph $G' = (V', E')$ is obtained from $G$ by a \textit{$K_4^{-}$-extension} if $G'$ is obtained from $G$ by the following process:
    \begin{itemize}
    \item[(i)] Delete an edge $v_1v_2 \in E$.
    \item[(ii)] Add two new vertices $u_1$ and $u_2$ so that $V' = V \cup \{u_1, u_2\}$.
    \item[(iii)] Add five new edges among $v_1, v_2, u_1,$and $u_2$ so that 
    $E' = (E \setminus \{v_1v_2\}) \cup \{v_1u_1, v_1u_2, v_2u_1, v_2u_2, u_1u_2\}$.
    \end{itemize}

    Let $G$ be a simple graph.
    A graph $G'$ is obtained from $G$ by a \textit{generalised vertex splitting} if $G'$ is obtained from $G$ by the following process:
    \begin{itemize}
    \item[(i)]Choose $v \in V$, a partition $N_0, N_1$ of the neighbours of $v$, and a vertex $x \in V(G) \setminus N_0$.
    \item[(ii)] Delete $v$ from $G$ and add two new vertices $v_0$ and $v_1$ joined to $N_0, N_1$, respectively so that $V(G') = (V(G) \setminus \{v\}) \cup \{v_0, v_1\}$.
    \item[(iii)] Add two new edges $v_0v_1$ and $v_0 x$.
    \end{itemize}

Our proof of Theorem~\ref{thm:sugiyama} uses
the following inductive construction due to Dewar, Hewetson, and Nixon~\cite{dewar2024uniquely}, which is built on a result by Jackson and Nixon~\cite{jackson2019}. 

\begin{theorem}[Dewar, Hewetson, and Nixon~\cite{dewar2024uniquely}]\label{thm:DHN (2,2)-connected}
Let $G$ be a graph. Then, $G$ is 2-connected and redundantly two-tree-connected if and only if $G$ can be generated from $K_5^{-}$ or $B_1$ depicted in Figure \ref{fig:base case} by $K_4^{-}$-extension, edge addition, and generalised vertex splits such that each intermediate graph $G'$ is 2-connected and redundantly two-tree-connected.
\end{theorem}

The following lemma solves Step 2 in the proof of Theorem~\ref{thm:sugiyama}.

\begin{lemma}\label{lem: base case of 2-dim}
    Let $G = K_5^{-}$ or $G = B_1$ and $(G, \bp)$ be a generic two-dimensional  framework.
    Then, $(G, \bp)$ has a self-stress $\omega$ such that $\rank L_{G, \omega^k} = |V(G)|-2$ for all $k$.
\end{lemma}
\begin{figure}
\centering
\begin{minipage}{0.4\textwidth}
\centering
\includegraphics[scale=0.6]{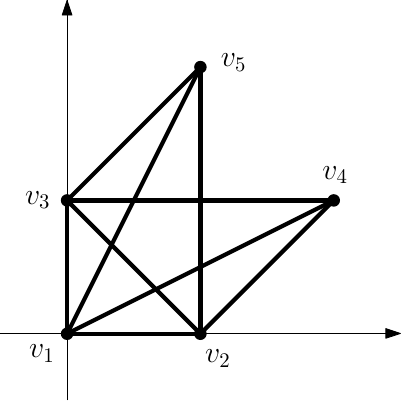}
    

\par
(a)
\end{minipage}
\begin{minipage}{0.4\textwidth}
\centering
\includegraphics[scale=0.6]{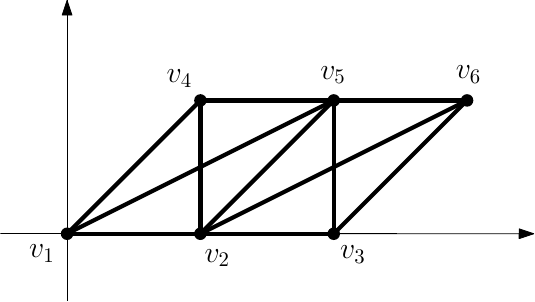}

\end{minipage}
\caption{(a) $K_5^-$ and (b) $B_1$.} \label{fig:base case}
\end{figure}

\begin{proof}
    By Proposition \ref{prop: avoiding generic construction}, it suffices to show that there exists a framework $(G, \bp)$ with
    $J(\pi_G\circ f_p^{\times d})(\bp)=2|V(G)|-2$ and having a $k$-th coordinated self-stress $\omega^k$ for some $k$ such that $\rank L_{G, \omega^k} = |V(G)| - 2$.
    We shall construct such $\bp$ explicitly for each of $K_5^-$ and $B_1$ as follows.
    
\medskip

\noindent 
{\bf Case 1.} 
Suppose  $G = K_5^{-}$.
    Let $(G, \bp)$ be defined by 
    \begin{equation*}
       \bp = (\bp(v_1), \ldots, \bp(v_5)) = ((0, 0), (1, 0), (0, 1), (2, 1), (1, 2)).
    \end{equation*}
    See Figure~\ref{fig:base case}(a). One can rapidly check that $\rank J(\pi_G\circ f_p^{\times 2})(\bp) = 8 = 2|V(G)|-2$ and 
    \begin{align*}
        \omega&=(\omega(v_1v_2), \omega(v_1v_3), \omega(v_1v_4), \omega(v_1v_5), \omega(v_2v_3), \omega(v_2v_4), \omega(v_2v_5), \omega(v_3v_4), \omega(v_3v_5))\\
        &= (2^p+2, 2^p+2, -2, -2, -2^p, 2, 2 - 2^{2-p}, 2 - 2^{2-p}, 2)
    \end{align*}
    is a unique self-stress of $(G, \bp)$ up to scalar multiplication. The Laplacian matrix weighted by the first-coordinated self-stress $\omega^1$ is 
    \begin{align*}
        L_{G, \omega^1} &= \begin{pmatrix}
            2^{p-1} & -2^p - 2 & 0 & 2^{p-1} & 2\\
            -2^p-2 & 4 & 2^p & -2 & 0\\
            0 & 2^p & 1 - 2^{p-1} & 1 - 2^{p-1} & -2\\
            2^{p-1} & -2 & 1 - 2^{p-1} & 1 & 0\\
            2 & 0 & -2 & 0 & 0
        \end{pmatrix}.
    \end{align*}
    By the triangular structure in the right-bottom three-by-three submatrix implies $\rank L_{G, \omega^1}\geq 3$.
    Since $\rank L_{G, \omega^1}\leq 3$, we obtain $\rank L_{G, \omega^1}=3$ for all $p$.

\medskip

\noindent 
{\bf Case 2.} 
Suppose $G = B_1$.
    Let $(G, \bp)$ be defined by 
    \begin{equation*}
        \bp = (\bp(v_1), \ldots, \bp(v_6))=
    ((0, 0), (1, 0), (2, 0), (1, 1), (2, 1), (3, 1)).
    \end{equation*}
    See Figure~\ref{fig:base case}(b). One can rapidly check that $\rank J(\pi_G\circ f_p^{\times 2})(\bp) = 10 = 2|V(G)|-2$ and 
    \begin{align*}
        \omega&=(\omega(v_1v_2), \omega(v_1v_4), \omega(v_1v_5), \omega(v_2v_3), \omega(v_2v_4), \omega(v_2v_5), \omega(v_2v_6), \omega(v_3v_5), \omega(v_3v_6), \omega(v_4v_5), \omega(v_5v_6))\\
        &= (2^{p-1}-1, 1, -1, -1, -1, 0, 1, 1, -1, 1, -2^{p-1}+1)
    \end{align*}
    is a unique self-stress of $(G, \bp)$ up to scalar multiplication.
    The Laplacian matrix weighted by the second coordinated self-stress $\omega^2$ is 
    \begin{align*}
         L_{G, \omega^2} &= \begin{pmatrix}
             0 & 0 & 0 & -1 & 1 & 0\\
             0 & 0 & 0 & 1 & 0 & -1\\
             0 & 0 & 0 & 0 & -1 & 1\\
             -1 & 1 & 0 & 0 & 0 & 0\\
             1 & 0 & -1 & 0 & 0 & 0\\
             0 & -1 & 1 & 0 & 0 & 0
         \end{pmatrix}
    \end{align*}
    and $\rank L_{G, \omega^2} = 4 = |V(G)| - 2$ for all $p$.
\end{proof}

Now we move to Step 3 in the proof of Theorem~\ref{thm:sugiyama}.
We first look at $K_4^{-}$-extension operations.
Our proof of the following lemma for $K_4^{-}$-extension operations is similar to the argument of Connelly~\cite{connelly2005generic} for the fact that a 1-extension operation preserves the stress condition in the Euclidean case.
Instead of using the dependency of a colinear triangle, we use the fact that $K_4$ placed on the vertices of an axis-aligned rectangle is dependent in $\ell_p^2$.

\begin{lemma}\label{lem:K4-extension}
Let $p$ be an integer with $p>2$ and $(G, \bp)$ and $(G', \bp')$ be generic locally rigid frameworks in $\ell_p^2$.
Suppose that $G'$ is a $K_4^{-}$-extension of $G$ along an edge $e$ and that  $(G, \bp)$ has a self-stress $\omega$ such that
$\omega(e)\neq 0$ and $\rank L_{G, \omega^{k}} = |V(G)| - 2$ for all $k$.
Then, $(G', \bp')$ has a self-stress $\omega'$ such that $\rank L_{G', {\omega'}^{k}} = |V(G')| - 2$ for all $k\in \{1,2\}$.
\end{lemma}
\begin{proof}
    Let  $V(G) = \{v_2, \ldots, v_n\}$ and suppose that $G'$ is constructed from $G$ by deleting $e=v_2v_3$, adding two new vertices $v_0, v_1$ and five new edges $v_0v_1, v_0v_2, v_0v_3, v_1v_2, v_1v_3$.
    By Proposition \ref{prop: avoiding generic construction}, it suffices to show that there exist a framework $(G', \bq)$ with $\rank  J(\pi_{G'}\circ f_p^{\times 2})(\bq) = 2|V(G')| - 2$ and a self-stress $\omega'$ of $(G',\bq)$ such that $\rank L_{G', {\omega'}^{k}} = |V(G')| - 2$ for some $k$.

    For each $v_i \in V(G)$, denote $\bp(v_i)=(x_i,y_i)$.
    We define $\bq:V(G')\rightarrow \mathbb{R}^2$ by  
    \begin{equation*}
        \bq(u) = \begin{cases}
            (x_0,y_0):=(x_2,y_3) & (u=v_0)\\
            (x_1,y_1):=(x_3,y_2) & (u=v_1)\\
            \bp(u) &(u \in V(G))
        \end{cases}
    \end{equation*}
    for $u\in V(G')$.
    Now consider the rows of $J(\pi_{G'}\circ f_p^{\times 2})(\bq)$ corresponding to the five new edges,
    \begin{equation*}
    \fontsize{4pt}{4pt}\selectfont
    \begin{blockarray}{ccccccc}
    v_0&v_1&v_2&v_3& v_4 & \cdots&v_n\\
        \begin{block}{(ccccccc)}
         (x_0-x_1)^{p-1}\ (y_0-y_1)^{p-1} & (x_1-x_0)^{p-1}\ (y_1-y_0)^{p-1} &0&0& 0 & \ldots&0\\
        (x_0-x_2)^{p-1}\ (y_0-y_2)^{p-1}  & 0 & (x_2-x_0)^{p-1}\ (y_2-y_0)^{p-1} &0&0 & \ldots&0\\
         (x_0-x_3)^{p-1}\ (y_0-y_3)^{p-1} &0&0& (x_3-x_0)^{p-1}\ (y_3-y_0)^{p-1} & 0 &\ldots&0\\
         0& (x_1-x_2)^{p-1}\ (y_1-y_2)^{p-1} & (x_2-x_1)^{p-1}\ (y_2-y_1)^{p-1} &0& 0 &\ldots&0\\
        0& (x_1-x_3)^{p-1}\ (y_1-y_3)^{p-1} &0& (x_3-x_1)^{p-1}\ (y_3-y_1)^{p-1} & 0 & \ldots&0\\
        \end{block}
    \end{blockarray}.
    \end{equation*}
    By $x_0=x_2, y_0=y_3, x_1=x_3, y_1=y_2$, the sum of the second, third, fourth, and fifth rows minus the first row is equal to 
    \begin{equation*}\footnotesize
        \begin{blockarray}{ccccccc}
        v_0&v_1&v_2&v_3& v_4 & \cdots&v_n\\
            \begin{block}{(ccccccc)}
             0 & 0 & (x_2-x_3)^{p-1}\ (y_2-y_3)^{p-1} & (x_3-x_2)^{p-1}\ (y_3-y_2)^{p-1} & 0 & \ldots&0 \\
            \end{block}
        \end{blockarray}
    \end{equation*}
    which is the row of $v_2v_3$.
    This implies that, by elementary operations, one can convert
    $J(\pi_{G'}\circ f_p^{\times 2})(\bq)$  to the following form
    \begin{equation*}
        \left(
        \begin{array}{c | c | c}
         {\begin{array}{cc} (x_0-x_2)^{p-1} &  (y_0-y_2)^{p-1}  \\ (x_0-x_3)^{p-1} & (y_0-y_3)^{p-1} \end{array}
        } & 
         O & \ast
        \\
        \hline
        O & 
        {\begin{array}{cc}(x_1-x_2)^{p-1} & (y_1-y_2)^{p-1} \\ 
        (x_1-x_3)^{p-1} & (y_1-y_3)^{p-1} \end{array}
        } & \ast
        \\
        \hline
        O & O & J(\pi_G\circ f_p^{\times 2})(\bp)
        \end{array}
        \right).
    \end{equation*}
    where the top four rows are corresponding to 
    $v_0v_2, v_0v_3, v_1v_2, v_1v_3$.
    Since the top-left and the middle two-by-two matrices are both non-singular, we obtain $\rank  J(\pi_{G'}\circ f_p^{\times 2})(\bq) = \rank  J(\pi_G\circ f_p^{\times 2})(\bp) + 4 =2|V(G')| - 2$.

    Next we define a self-stress $\omega'$ of $(G',\bq)$.
    Let $\omega$ be a generic self-stress of $(G,\bp)$,
    and define  $\omega'$ by
    \begin{equation*}
        \omega'(e) = \begin{cases}
            -\omega(v_1v_3) & (e = v_0v_1) \\
            \omega(v_2v_3) & (e \in \{v_0 v_2, v_0v_3, v_1v_2, v_1 v_3\})\\
            \omega(e) & (e \in E \setminus (v_2 v_3))
        \end{cases}
    \end{equation*}
    for $e\in E(G')$.
    Since the sum of the rows of $v_0v_2, v_0v_3, v_1v_2, v_1v_3$ minus that of $v_0v_1$ leads to the row of $v_2v_3$, $\omega'$ is a self-stress of $(G',\bq)$.

   Now, we look at the first-coordinated self-stress of $\omega'$. 
   By $x_0=x_2$ and $x_1=x_3$, 
   we have $\omega'^1(v_0v_1)=\omega'(v_0v_1)(x_0-x_1)^{p-2}=-\omega(v_2v_3)(x_2-x_3)^{p-2}=-\omega^1(v_2v_3)$.
   Similarly, we have 
   $\omega'^1(v_0v_2)=\omega^1(v_2v_3)$, 
   $\omega'^1(v_0v_3)=0$, 
   $\omega'^1(v_1v_2)=0$, 
   $\omega'^1(v_1v_3)=\omega^1(v_2v_3)$.
   Therefore,  the Laplacian matrix $L_{G', \omega'^1}$ is given by
    \begin{align}\label{eq:K4ex}
        \left(
         \begin{array}{c|c}
            \begin{array}{cc}
                0 & 0 \\
                0 & 0 
            \end{array}&  O\\
             \hline
             O & L_{G, \omega^1}
         \end{array}
        \right)
        +
        \left(\begin{array}{c | c}
        \begin{array}{cccc}
                0 & \omega^1(v_2v_3) & - \omega^1(v_2v_3) & 0\\
                \omega^1(v_2v_3) & 0 & 0 & -\omega^1(v_2v_3) \\
                - \omega^1(v_2v_3) & 0 & 0 & \omega^1(v_2v_3) \\
                0 & - \omega^1(v_2v_3) & \omega^1(v_2v_3) & 0 
            \end{array}
        & O \\
       \hline
       O & O
        \end{array}
        \right).
    \end{align}
    Since $\omega(v_2v_3)=\omega(e) \neq  0$ by the lemma assumption, $\omega^1(v_2v_3) \neq  0$ and hence the top-left two-by-two submatrix 
    $\begin{pmatrix} 0 & \omega^1(v_2v_3) \\\omega^1(v_2v_3) & 0 \end{pmatrix}$ 
    is non-singular. Moreover, the Schur complement of the top-left two-by-two submatrix in the right matrix in (\ref{eq:K4ex}) is the zero matrix.
    Hence, the Schur complement of the top-left two-by-two submatrix in $L_{G',\omega'^1}$ results in $L_{G,\omega^1}$.  
    Thus, $\rank L_{G', \omega'^1} = \rank L_{G, \omega^1} + 2 = |V(G')| - 2$.
    This completes the proof.
\end{proof}

The proof of the corresponding statement for generalised vertex splitting is more involved.
We introduce special notations to simplify the presentation.

For a two-dimensional framework $(G,\bp)$
and an edge weight $\omega:E(G)\rightarrow \mathbb{R}$, 
the $k$-th weighted direction vector $\bd_{\bp,\omega}^k(u,v)$ from $u$ to $v$ for $uv\in E$ is defined by 
\[
\bd_{\bp,\omega}^k(u,v)=\omega(uv) (\bp(v)-\bp(u))^k\in \mathbb{R}^2.
\]
Observe that the equilibrium condition at a vertex $u$ is written as 
$\sum_{v\in N_G(u)} \bd_{\bp,\omega}^{p-1}(u,v)=0$,
where $N_G(u)$ denotes the set of neighbours of $u$ in $G$.

\begin{lemma}\label{lem: generalised vertex split}
Let $G$ and $G'$ be 2-connected and redundantly 2-tree-connected graphs such that $G'$ is obtained from $G$ by a generalised vertex splitting.
Suppose a generic two-dimensional framework $(G, \bp)$ has a self-stress $\omega$ such that $\rank L_{G, \omega^{k}} = |V(G)| - 2$ for all $k \in \{1,2\}$.
Then, a generic two-dimensional framework $(G', \bp')$ has a self-stress $\omega'$ such that $\rank L_{G', {\omega'}^{k}} = |V(G')| - 2$ for all $k \in \{1,2\}$.
\end{lemma}
\begin{proof}
    Let $V(G) = \{v_1, \ldots, v_n\}$ and suppose that $G'$ is constructed from $G$ by first partitioning $N_G(v_1)$ into $N_0$ and $N_1$, replacing the edges $v_1u$ with $v_0u$ for $u\in N_0$, and adding $v_0 v_1$ and $v_0 v_2$ for some $v_2 \in V(G) \setminus (N_G(v_1) \setminus N_0)$.
    An edge weight $\omega$ of $G$ can be considered as an edge weight $\omega'$ of $G'$ by setting $\omega'(e)=\omega(e)$ if $e\in E(G')\cap E(G)$, $\omega'(v_0u)=\omega(v_1u)$ for $u\in N_0$, and $\omega'(v_0v_1)=\omega'(v_0v_2)=0$.
    By this convention, in the subsequent discussion we shall sometimes regard an edge weight of $G$ as that of $G'$.
    
    Pick a generic framework $(G,\bp)$ and a generic self-stress $\omega^*$ of $(G,\bp)$.
    Set a vector
    \begin{equation}\label{eq:d*}
    \bd^*:=\sum_{v_j\in N_0} \bd_{\bp,\omega^*}^{p-1}(v_j,v_1),
    \end{equation}
    and, for each real parameter $t$, 
    define a framework $(G',\hat{\bp}_t)$ by 
    \begin{equation*}
        \hat{\bp}_t(u) = \begin{cases}
            \bp(v_1) + (t \bd^* )^{\frac{1}{p-1}} &(u = v_0)\\
            \bp(u) & (u \in V(G))
        \end{cases}
    \end{equation*}
    for $u\in V(G')$.
    (Recall our convention of the $k$-th power notation for vectors.)
    We first investigate properties of $\bd^*$ and $(G',\hat{\bp}_0)$.
    \begin{claim}\label{claim:1}
    $(G'-v_0v_1,\hat{\bp}_0)$ is infinitesimally rigid
    and $\bd^*$ is a non-zero vector.
    \end{claim}
    \begin{proof}
    This is a direct consequence of the $uv$-coincidence rigidity characterisation due to Dewar, Hewetson, and Nixon~\cite{sean2024}.
    A graph $H$ is said to be {\em $uv$-coincident rigid} in $\ell_p^2$
    if there is an infinitesimally rigid framework $(H-uv,\bq)$ in $\ell_p^2$ such that $\bq(u)=\bq(v)$.
    Dewar, Hewetson, and Nixon~\cite{sean2024} showed that 
    a graph $H$ is $uv$-coincident rigid in $\ell_p^2$ if and only if 
    $H-uv$ and $H/uv$ are rigid in $\ell_p^2$,
    where $H/uv$ denotes the graph obtained from $H$ by identifying $u$ and $v$.

    Back to our problem, $G'-v_0v_1$ is rigid since $G'$ is redundantly rigid, and $G'/v_0v_1$ is rigid since $G'/v_0v_1=G+v_1v_2$ and $G$ is rigid.
    Hence, $G'$ is $v_0v_1$-coincident rigid,
    and so is $G'-v_0v_1$.
    Since $(G',\hat{\bp}_0)$ is a realisation of $G'$ such
    that $\hat{\bp}_0(v_0)=\hat{\bp}_0(v_1)$
    and the restriction of $\hat{\bp}_0$ to $V(G)$ is generic,
    $(G'-v_0v_1,\hat{\bp}_0)$ is infinitesimally rigid.

    To see $\bd^*\neq 0$, assume $\bd^*=0$.
    Then (using our conversion of converting an edge weight of $G$ to that of $G'$) we claim 
    \begin{equation}\label{eq:claim12}
    \text{every self-stress $\omega$ of $(G,\bp)$ is a self-stress of $(G'-v_0v_1-v_0v_2,\hat{\bp}_0)$.}
    \end{equation}
    Indeed, since $\bd^*=0$ is a linear relation in the entries of $\omega^*$ (cf.~(\ref{eq:d*}))
    and $\omega^*$ is taken to be generic, $\bd^*=0$ implies 
    that 
    $\sum_{u\in N_0} d_{\bp, \omega}^{p-1}(u,v_1)=0$
    holds
    for any self-stress $\omega$ of $(G,\bp)$.
    Also, the equilibrium condition at $v_1$ in $(G,\bp)$ gives
    $\sum_{u\in N_G(v_1)} d_{\bp, \omega}^{p-1}(u,v_1)=0$,
    which combined with the last relation gives
    $\sum_{u\in N_G(v_1)\setminus N_0} d_{\bp, \omega}^{p-1}(u,v_1)=0$.
    By $N_{G'-v_0v_1-v_0v_2}(v_0)=N_0$ and $N_{G'-v_0v_1-v_0v_2}(v_1)=N_G(v_1)\setminus N_0$, these are written as 
    \begin{equation*}
    \sum_{u\in N_{G'-v_0v_1-v_0v_2}(v_0)} d_{\hat{\bp}_0, \omega}^{p-1}(u,v_1)=0 \quad \text{and} \quad
    \sum_{u\in N_{G'-v_0v_1-v_0v_2}(v_1)} d_{\hat{\bp}_0, \omega}^{p-1}(u,v_1)=0,
    \end{equation*}
    and thus $\omega$ satisfies the equilibrium condition at $v_0$ and $v_1$ in $(G'-v_0v_1-v_0v_2,\hat{\bp}_0)$.
    The equilibrium condition at other vertices is a straightforward consequence  of the definition of $\hat{\bp}_0$ and the fact that $\omega$ is a self-stress of $(G,\bp)$,
    and hence (\ref{eq:claim12}) follows.

    Now (\ref{eq:claim12}) gives
    $\dim \ker J(\pi_G\circ f_p^{\times 2})(\bp)^{\top} 
    \leq \dim \ker J(\pi_{G'-v_0v_1-v_0v_2}\circ f_p^{\times 2})(\hat{\bp}_0)^{\top} 
    \leq \dim \ker J(\pi_{G'-v_0v_1}\circ f_p^{\times 2})(\hat{\bp}_0)^{\top}$.
    Also the rigidity of $G$ gives 
    $\dim \ker J(\pi_G\circ f_p^{\times 2})(\bp)^{\top}=|E(G)|-(2|V(G)|-2)$.
    Hence, by $|E(G'-v_0v_1)|=|E(G)|+1$ and $|V(G')|=|V(G)|+1$,
    $\rank J(\pi_{G'-v_0v_1}\circ f_p^{\times 2})(\hat{\bp}_0)
    = |E(G'-v_0v_1)|-\dim \ker J(\pi_{G'-v_0v_1}\circ f_p^{\times 2})(\hat{\bp}_0)^{\top}
    \leq |E(G'-v_0v_1)|- (|E(G)|-(2|V(G)|-2))
    =2|V(G)|-3$.
    This contradicts the infinitesimal rigidity of $(G'-v_0v_1,\hat{\bp}_0)$ proved in the first part.
    \end{proof}

    By the definition of $\hat{\bp}_t$, the row of $v_0 v_1$ in $J(\pi_{G'}\circ f_p^{\times 2})(\hat{\bp}_t)$ is given as
    \begin{equation}\label{eq:v0v1}
    \begin{blockarray}{ccccc}
    v_0&v_1&v_2&  \cdots&v_n\\
        \begin{block}{(ccccc)}
         t\bd^* & -t\bd^* & 0 & \ldots&0 \\
        \end{block}
    \end{blockarray}
    \end{equation}
    Let $R_t$ be the matrix in which the row corresponding to $v_0 v_1$ in $J(\pi_{G'}\circ f_p^{\times 2})(\hat{\bp}_t)$ is divided by $t$.
    Note that $R_0$ is well-defined by (\ref{eq:v0v1}).
    Since $(G'-v_0v_1,\hat{\bp}_0)$ is infinitesimally rigid by Claim~\ref{claim:1}, $\rank R_0=2|V(G')|-2$.
    Since the entries of $R_t$ are rational functions in $t$,
    there exists a real number $\bar{t}>0$ such that 
    $\rank R_t=2|V(G')|-2$ for all $t$ with $0\leq t < \bar{t}$.
    This also implies that 
    $\rank J(\pi_{G'}\circ f_p^{\times 2})(\hat{\bp}_t)=2|V(G')|-2$ 
    for any $t$ with $0<t<\bar{t}$.


    Next, we define an edge weight $\bar{\omega}_0:E(G')\rightarrow \mathbb{R}$ by 
    \begin{equation}\label{eq:last2}
        \bar{\omega}_0(e) = \begin{cases}
            -1 & (e = v_0 v_1)\\
            0 & (e = v_0 v_2)\\
            \omega^* (e) & (e \in E(G')\setminus \{v_0v_1, v_0v_2\}).
        \end{cases}
    \end{equation}
    With this definition, $\bar{\omega}_0$ is in the left kernel of $R_0$
    by the definition (\ref{eq:d*}) of $\bd^*$ and the fact that $\omega^*$ is a self-stress of $(G,\bp)$.
    Since the entries of $R_t$ are rational functions in $t$ 
    and the rank of $R_t$ is invariant over $0 \le t < \bar{t}$, 
    we can extend $\bar{\omega}_0$ to a family of maps $\bar{\omega}_t: E(G')\rightarrow \mathbb{R}$ continuous in $t$ over $[0, \bar{t})$ such that $\bar{\omega}_t$ is in the left kernel of $R_t$.
    Recall that $J(\pi_{G'}\circ f_p^{\times 2})(\hat{\bp}_t)$ is obtained from $R_t$ by multiplying the row of $v_0v_1$ by $t$.
    So, if we define $\hat{\omega}_t: E(G')\rightarrow\mathbb{R}$ by 
    \begin{equation}\label{eq:hat_omega}
    \hat{\omega}_t(e)=\begin{cases}
    \frac{1}{t}\bar{\omega}_t(e) & (e=v_0v_1) \\
    \bar{\omega}_t(e) & (e\in E(G')\setminus \{v_0v_1\}),
    \end{cases}
    \end{equation}
    then we obtain a family of maps $\hat{\omega}_t$ continuous in $t$ over $(0, \bar{t})$ such that $\hat{\omega}_t$ is in the left kernel of $J(\pi_{G'}\circ f_p^{\times 2})(\hat{\bp}_t)$,
    i.e., $\hat{\omega}_t$ is a self-stress of $(G',\hat{\bp}_t)$.
    
    It remains to prove $\rank L_{G',\hat{\omega}_t^1}=|V(G')|-2$
    for the first coordinated stress $\hat{\omega}_t^1$ for all sufficiently small $t>0$.
    For this, a key step is to relate $L_{G',\hat{\omega}_t^1}$
    with $L_{G,\omega^{*1}}$ (since we know $\rank L_{G,\omega^{*1}}=|V(G)|-2$ by the lemma assumption).
    Our strategy for establishing this relationship is to look at the Shur complement at the diagonal of $v_0$ in   $L_{G',\hat{\omega}_t^1}$.
    Hence, we first need to show that the diagonal at $v_0$ in  $L_{G',\hat{\omega}_t^1}$ is nonzero, which is the stratement of the next claim.
    
    \begin{claim}\label{claim:3}
    Let $\varepsilon_t=L_{G', \hat{\omega}_t^{1}}[v_0, v_0]$.
    Then $\lim_{t \to 0} \frac{1}{\varepsilon_t}=0$.
    \end{claim}
    \begin{proof}
    By the equilibrium condition of $\hat{\omega}$ at $v_0$ in $(G',\hat{\bp}_t)$, 
    $
    \bd_{\hat{\bp}_t,\hat{\omega}}^{p-1}(v_0,v_1)+
    \sum_{u\in N_0\cup \{v_2\}} \bd_{\hat{\bp}_t,\hat{\omega}}^{p-1}(v_0,u)=0.
    $
    By substituting the definition of $\hat{\bp}_t$, 
    \begin{equation}\label{eq:last1}
    \hat{\omega}_t(v_0 v_1) t \bd^*
         +\sum_{u \in N_0\cup \{v_2\}}\hat{\omega}_t(v_0 u)\left(\bp(u)  - \bp(v_0) - (t\bd^*)^{\frac{1}{p-1}}\right)^{p-1} = 0.
    \end{equation}
    By (\ref{eq:last2}) and (\ref{eq:hat_omega}), $\lim_{t\to 0} \hat{\omega}_t(v_0v_2)=\bar{\omega}_0(v_0v_2)=0$
    and $\lim_{t\to 0} \hat{\omega}_t(v_0u)=\bar{\omega}_0(v_0u)=\omega(v_0u)$
    for each $u\in N_0$.
    Hence, letting $t \to 0$, the second term of the left side of (\ref{eq:last1}) converges to
    $\bd^*$ .
    Hence, by (\ref{eq:last1}), 
    \begin{equation}\label{eq:last4}
        \lim_{t \to 0}t \hat{\omega}_t(v_0 v_1) = -1.
    \end{equation}

    Now we look at $\varepsilon_t:=L_{G', \hat{\omega}_t^{1}}[v_0,v_0]$.
    For each $u\in V(G')$, let $\hat{\bx}_t(u)$ be the first entry of $\hat{\bp}_t(u)$. Similarly, denote $\bx(u)$ to be the first entry of $\bp(u)$.
    Since $L_{G',\hat{\omega}_t^1}$ is a Laplacian,
    \begin{align}\nonumber
        \varepsilon_t=L_{G', \hat{\omega}_t^{1}}[v_0, v_0] &=
        \sum_{u\in N_0\cup\{v_1,v_2\}} \hat{\omega}_t^1(v_0u)\\
        & = \hat{\omega}_t(v_0 v_1)(\hat{\bx}_t(v_1) - \hat{\bx}_t(v_0))^{p-2}
        + \sum_{u \in N_0\cup \{v_2\}}\hat{\omega}_t(v_0 v_i)(\hat{\bx}_t(u) - \hat{\bx}_t(v_0))^{p-2}. \label{eq:last3}
    \end{align}
    Since each entry of $\hat{\bx}_t$ converges to a finite vector when $t\to 0$ and $\lim_{t\to 0} \hat{\omega}_t(v_0v_2)=0$ and $\lim_{t\to 0} \hat{\omega}_t(v_0u)=\omega(v_0u)$,
    the second term in the right side of (\ref{eq:last3}) converges to a finite value when $t\to 0$.
    On the other hand, the first term of the right side of (\ref{eq:last3}) is 
    \begin{equation}\label{eq:last5}
        \hat{\omega}_t(v_0 v_1)(\hat{\bx}_t(v_1)- \hat{\bx}_t(v_0))^{p-2}
        = \hat{\omega}_t(v_0 v_1) (td_1^*)^{\frac{p-2}{p-1}} = (t \hat{\omega}_t(v_0 v_1))t^{-\frac{1}{p-1}} (d_1^*)^{\frac{p-2}{p-1}},
    \end{equation}
    where $d_1^*$ denotes the first entry of $\bd^*$.
    By (\ref{eq:last4}), the right side of (\ref{eq:last5}) diverges when $t\to 0$.
    Hence, by (\ref{eq:last3}), $\varepsilon_t$ also diverges when $t\to 0$.
    This completes the proof.
    \end{proof}
    
    Denote 
    \[
    s_t:= \sum_{u \in N_0\cup \{v_2\}} L_{G',{\hat{\omega}_t}^{1}}[v_0,u]=\sum_{u \in N_0\cup \{v_2\}}-\hat{\omega_t}^{1}(v_0 u).
    \]
    Since $L_{G',{\hat{\omega}_t}^{1}}$  is a Laplacian, the sum of the entries of the row of $v_0$ is equal to zero, which gives 
    \[
    L_{G',{\hat{\omega}_t}^{1}}[v_0, v_1]=-\varepsilon_t-s_t.
    \]
    Also, denote by $\gamma_t$ the vector of dimension $|N_0\cup \{v_2\}|$ obtained by aligning 
    $L_{G',{\hat{\omega}_t}^{1}}[v_0,u]$ for $u\in N_0\cup \{v_2\}$.
    Then $L_{G',{\hat{\omega}_t}^{1}}$ can be written as
    \begin{equation}\label{eq: laplacian of G'}
       L_{G', \hat{\omega}_t^{1}}
       = \left(
         \begin{array}{c|c}
            0&  O\\
             \hline
             O & L_{G' - v_0, \hat{\omega}_t^{1}}
         \end{array}
        \right)
       + \begin{blockarray}{ccccc}
           \, & v_0 & v_1  & N_0\cup\{v_2\} & Y\\
           \begin{block}{c(ccc|c)}
               v_0 & \varepsilon_t & -\varepsilon_t - s_t & \gamma_t &  0\\
               v_1 & -\varepsilon_t - s_t & \varepsilon_t + s_t & 0  & 0 \\
               N_0\cup\{v_2\} & \gamma_t^{\top}  &  0 & -\mathrm{diag}(\gamma_t) & 0  \\\cline{2-5}
               Y   & 0 & 0 & 0  & 0 \\
           \end{block}
       \end{blockarray}
    \end{equation}
    where the first term is the contribution of the edges in $G'-v_0$ whereas the second term is the contribution of the edges incident to $v_0$ in $G'$.
    Note that $s_t$ and $\gamma_t$ converge to a finite value and a finite vector, respectively, when $t\to 0$.
    Hence, if we take the Shur complement of (\ref{eq: laplacian of G'}) at the left-top corner, then the resulting matrix $L_{G', {\omega'_t}^{1}}/v_0$ becomes
    \begin{align}\label{eq:last6} 
        L_{G', \hat{\omega}_t^{1}}/v_0
        &= L_{G' - v_0, \hat{\omega}_t^{1}}
        +\begin{blockarray}{cccc}
            \, & v_1 &   N_0\cup\{v_2\} & Y \\
            \begin{block}{c(cc|c)}
            v_1 & -s_t &  \gamma_t & 0\\
            N_0\cup\{v_2\} & \gamma_t &  -\mathrm{diag}(\gamma_t) & 0 \\\cline{2-4}
            Y & 0 & 0  & 0 \\
            \end{block}
        \end{blockarray}+ O\left(\frac{1}{\varepsilon_t}\right).
    \end{align}
    We now show that $L_{G', \hat{\omega}_t^{1}}/v_0$ converges to  $L_{G,\omega^{*1}}$ when $t\to 0$.
    To see this, observe first that, $G'-v_0$ is isomorphic to the graph $H$ obtained from $G$ by removing the edges $v_1u$ for all $u\in N_0$.
    Hence, by (\ref{eq:last2}), 
    the first term of the right side of (\ref{eq:last6}) converges to $L_{H,\omega^{*1}}$.
    The second term of the right side of (\ref{eq:last6}) converges to 
    the Laplacian corresponding to the edges $v_1u$ for all $u\in N_0$
    since 
    $\lim_{t\to 0} -s_t=\sum_{u\in N_0\cup \{v_2\}} \bar{\omega}_0^1(v_0u)=\sum_{u\in N_0} \omega^{*1}(v_0u)$ by (\ref{eq:last2})
    and the fact that the entry of $\gamma_t$ indexed by $u\in N_0\cup \{v_2\}$ converges to $L_{G,\omega^{*1}}[v_1,u]$ if $u\in N_0$ or zero if $u=v_2$ by (\ref{eq:last2}).
    Therefore, by Claim~\ref{claim:3}, the right side of (\ref{eq:last6}) converges to $L_{G,\omega^{*1}}$ when $t\to 0$,
    and we obtain 
    \begin{equation}\label{eq:last7}
    \lim_{t\to 0} L_{G', \hat{\omega}_t^{1}}/v_0 =L_{G,\omega^{*1}}.
    \end{equation}

    By the lemma assumption, $\rank L_{G,\omega^{*1}}=|V(G')|-2$.
    Hence, by (\ref{eq:last7}), $\rank L_{G', \hat{\omega}_t^{1}}/v_0=|V(G')|-2$ for any sufficiently small $t>0$.
    Claim~\ref{claim:3} also implies that $L_{G', \hat{\omega}_t^{1}}[v_0,v_0]=\varepsilon_t$ is nonzero for any sufficiently small $t>0$. 
    Thus, $\rank L_{G', \hat{\omega}_t^{1}}=\rank L_{G', \hat{\omega}_t^{1}}/v_0+1=|V(G)|-1=|V(G')|-2$ for any sufficiently small $t>0$.
    This completes the proof.
    \end{proof}

We are now ready to complete the proof of Theorem~\ref{thm:sugiyama}.

\begin{proof}[Proof of Theorem \ref{thm:sugiyama}]
Suppose $G$ is 2-connected.
Then, by Theorem \ref{thm:DHN (2,2)-connected}, $G$ can be constructed inductively 
from $K_5^-$ or $B_1$ by $K_4$-extensions and generalised vertex splittings keeping 
2-connectivity and redundant 2-tree-connectivity.
For $K_5^-$ and $B_1$, the theorem holds by Lemma~\ref{lem: base case of 2-dim}.
In general, the theorem holds by Lemmas~\ref{lem:K4-extension} and~\ref{lem: generalised vertex split}.

Finally, we consider a general case.
The proof is identical to that of Theorem~\ref{thm:laplacian of 1-dim}.
Note that (redundant) 2-tree-connectivity implies 2-edge-connectivity.
\end{proof}

Based on Theorem \ref{thm:sugiyama}, we now complete the proof of our main theorem, Theorem~\ref{thm:main}.
We actually prove the following stronger statement.
\begin{theorem}\label{thm:last_main}
Let $p$ be an even positive integer with $p\neq 2$, $n\geq 3$ be a positive integer, and $G$ be a connected graph with $n$ vertices.
Then the following are equivalent.
\begin{itemize}
\item[(i)] A/every generic 2-dimensional framework $(G,\bp)$ is globally rigid in $\ell_p^2$.
\item[(ii)] $\overline{\pi_G({\cal CM}^p_n)}$ is $2$-identifiable.
\item[(iii)] $\overline{\pi_G({\cal CM}^p_n)}$ is not  2-tangentially weakly defective.
\item[(iv)] A generic 2-dimensional framework $(G,\bp)$ has a $k$-th coordinated self-stress $\omega^k$ such that $\rank L_{G,\omega^k}=n-2$ for all $k$.
\item[(v)] $G$ is 2-connected and redundantly two-tree-connected.
\end{itemize}
\end{theorem}
\begin{proof}
(i) implies (v) by Theorem~\ref{thm:GHT}.
(v) implies (iv) by Theorem~\ref{thm:sugiyama}.
(iv) implies (iii) by Lemma~\ref{lem:COV_test}.
(iii) implies (ii) by Proposition~\ref{prop:twnd_iden}.
(ii) implies (i) by Proposition~\ref{prop:iden_global}.
\end{proof}

An obvious open problem is to extend the result to the  general dimensional case.
\begin{conj}\label{conj:last_main}
Let $p$ be an even positive integer with $p\neq 2$, $n\geq 3$ be a positive integer, and $G$ be a connected graph with $n$ vertices.
Then the following are equivalent.
\begin{itemize}
\item[(i)] A/every generic $d$-dimensional framework $(G,\bp)$ is globally rigid in $\ell_p^d$.
\item[(ii)] $\overline{\pi_G({\cal CM}^p_n)}$ is $d$-identifiable.
\item[(iii)] $\overline{\pi_G({\cal CM}^p_n)}$ is not $d$-tangentially weakly defective.
\item[(iv)] A generic $d$-dimensional framework $(G,\bp)$ has a $k$-th coordinated self-stress $\omega^k$ such that $\rank L_{G,\omega^k}=n-2$ for all $k$.
\item[(v)] $G$ is 2-connected and redundantly $d$-tree-connected.
\end{itemize}
\end{conj}

Another natural open problem is to give a characterization in the case where $p$ is odd.
As we remarked at the end of Section~\ref{subsec:cayley-menger},
there is a gap between the $p$-Cayley Menger variety defined in 
Section~\ref{subsec:cayley-menger}) and the image of the $\ell_p$-distance measurement map when $p$ is odd,
 and we are currently missing Theorem~\ref{thm:necessity} and  Proposition~\ref{prop:iden_global}.
Extending the techniques developed in this paper to settings where a rigidity map is a piecewise polynomial (including the $\ell_p^d$-distance measurement map) is an interesting open problem.
See~\cite{dewar2025uniquely} for recent developments concerning polynomial-norm measurement maps.

\paragraph{Acknowledgement} 
We thank the anonymous referee for several suggestions that improved the paper.
We also thank D{\'a}niel Garamv{\"o}lgyi for his careful reading of the manuscript and constructive comments.
The work was supported by 
JST PRESTO Grant Number JPMJPR2126.

\bibliographystyle{abbrv}
\bibliography{reference.bib}

\end{document}